\documentclass[10pt]{amsart}
\allowdisplaybreaks
\usepackage[pagewise]{lineno}
\usepackage{fullpage}
\usepackage{amsfonts}
\usepackage{amsmath}
\usepackage{amssymb}
\usepackage{amsthm}
\usepackage{pict2e}
\usepackage{hyperref}
\usepackage{amsbsy}
\usepackage{xcolor}
\usepackage{graphicx} \usepackage{enumerate} \usepackage{multicol}
\usepackage{mathrsfs} \usepackage[all,cmtip]{xy}
\usepackage[color=blue!20,textsize=footnotesize]{todonotes}

\newcounter{dummy} \numberwithin{dummy}{section}
\newtheorem{theorem}[dummy]{Theorem}
\newtheorem{corollary}[dummy]{Corollary}
\newtheorem{lemma}[dummy]{Lemma}
\newtheorem{definition}[dummy]{Definition}
\newtheorem{proposition}[dummy]{Proposition}
\theoremstyle{remark}
\newtheorem{remark}[dummy]{Remark}
\newtheorem{example}[dummy]{Example}

\newcommand{\calA}{\mathcal{A}}
\newcommand{\calI}{\mathcal{I}}
\newcommand{\calP}{\mathcal{P}}
\newcommand{\calT}{\mathcal{T}}
\newcommand{\calC}{\mathcal{C}}
\newcommand{\calE}{\mathcal{E}}

\DeclareMathOperator{\id}{id}

\DeclareMathOperator{\pr}{pr}
\DeclareMathOperator{\rank}{rank}
\DeclareMathOperator{\image}{image}
\DeclareMathOperator{\spn}{span}

\DeclareMathOperator{\frakg}{\mathfrak{g}}

\DeclareMathOperator{\SU}{SU}
\DeclareMathOperator{\SL}{SL}

\DeclareMathOperator{\gr}{gr}

\numberwithin{equation}{section}

\title[Filtered complexes and cohomologically equivalent subcomplexes]{Filtered complexes and cohomologically equivalent subcomplexes}
\author[E.~Grong and F.~Tripaldi]{Erlend Grong and Francesca Tripaldi}

\thanks{The first author is supported by the grant GeoProCo from the Trond Mohn Foundation - Grant TMS2021STG02 (GeoProCo). The second author was partially supported by the Swiss National Science Foundation Grant nr 200020-191978, \textit{Analytic and geometric structures in singular spaces} and would also like to thank the Centro di
Ricerca Matematica Ennio De Giorgi and the Scuola Normale Superiore for the hospitality and support.
}

\subjclass[2020]{58G05, 58A10, 53C15, 58A12}

\address{University of Bergen, Department of Mathematics, P.O.~Box 7803, 5020 Bergen, Norway}
\email{erlend.grong@uib.com}

\address{University of Leeds,
Department of Pure Mathematics, Leeds, UK}
\email{f.tripaldi@leeds.ac.uk}

\keywords{Filtered differential complex, base differential, Rumin complex, spectral sequences.}

\begin{document}

\begin{abstract}
Inspired by Rumin's work on a subcomplex in sub-Riemannian manifolds that is cohomologically equivalent to the de Rham complex, we present a more general construction that produces subcomplexes from any filtered cochain complex of finite depth and still computes the cohomology of the original filtered complex. A priori, these subcomplexes depend not only on the filtration itself, but also on the choice of additional structures. However, we show that the construction depends only on the given filtration up to isomorphism. Finally, we show how such subcomplexes relate to spectral sequences, a cohomological machinery that arises naturally when considering a filtered complex. 
\end{abstract}

\maketitle

\section{Introduction} \label{sec:Introduction}
Filtered complexes show up all over mathematics, from differential forms \cite{Hat60,BotTu82}, Lie algebra cohomology \cite{HoSe53}, algebraic geometry \cite{Abd00}, K-theory \cite{Bloch77}, filtered Riemannian manifolds \cite{Sak78,Alv89}, sub-Riemannian manifolds \cite{Rumin1,Rumin2,Rumin4}, BGG-sequences \cite{CSS01,dave2020heat,DaHa22} and parabolic geometries \cite{CaSl09} just to name a few. In 1994, Rumin \cite{Rumin1,Pan93} pointed out that the de Rham complex $(\Omega^\bullet,d)$ of a contact manifold contains a subcomplex that is isomorphic at the level of cohomology. The key idea behind this construction is the introduction of a tensorial map $d_0\colon\Omega^\bullet\to\Omega^\bullet$ such that $d-d_0$ becomes a nilpotent operator. This idea has been further developed in \cite{Rumin2,Rumin3,Rumin4} into a complex that is well defined on any sub-Riemannian manifold, once some preliminary choices for the grading and metric are made. Such subcomplex, which is now known as the Rumin complex, has proven particularly useful in the sub-Riemannian context for the study of holonomy of partial connections \cite{FGR97,GrPa18,CGJK19}, analytic torsions \cite{RuSe12,Kit20,Hall22}, invariants of PDEs on differential forms \cite{Rumin4,AlQu22,dave2020heat}, sub-Laplacians and geometric invariants on CR-manifold \cite{AGL02,BHR07}, and Sobolev inequalities and $L^{p,q}$-cohomologies \cite{BFT06,PaRu18,PaTr19,BTT22,BFP22}.

In the current paper, we demonstrate how the Rumin complex appears as a specific instance of a more general method of generating subcomplexes that can be applied to any filtered cochain complex $(K,d)$. Assume that the filtration $K= K_0 \supset K_1 \supset K_2 \supset \cdots \supset K_S \supset K_{S+1} =0$ has finite depth $S$ such that $dK_j \subset K_j$, $j =0, \dots, S$. We define the graded complex
\begin{equation} \label{DefGr}
\gr(K) = K_0/K_1 \oplus K_1 /K_2 \oplus \cdots \oplus K_S/K_{S+1} = \{ \oplus_j(\alpha_j + K_{j+1}) \, :\, \alpha_j \in K_j, j=0, \dots, S \}
\end{equation}
with differential $d_{\gr} \oplus_j (\alpha_j + K_{j+1}) = \oplus_j (d\alpha_j + K_{j+1})$. \emph{A splitting} of this filtration is a direct sum $K = K_{(0)} \oplus K_{(1)} \oplus \cdots \oplus K_{(S)}$ such that $K_{j} = K_{(j)} \oplus K_j$. Such a splitting can equivalently be described as an invertible linear map $\phi: K \to \gr(K)$ such that $\phi(K_j) = K_j/K_{j+1} \oplus \cdots \oplus K_S/K_{S+1} \subset \gr(K)$ such that the inverse image of $K_j/K_{j+1}$ under $\phi$ is a complement to $K_{j+1}$ of $K_j$. The correspondence between this structure are such that if $\alpha = \sum_{j=0}^S \alpha_{(j)}$ with $\alpha_{(j)} \in K_{(j)}$, then $\phi(\alpha) = \oplus_j (\alpha_{(j)}+K_j)$.
The objective of this paper is to find a differential $d_H: H(\gr(K))  \to H(\gr(K))$ on the cohomology group $H(\gr(K)) =(\ker d_{\gr})/(\image d_{\gr})$ such that closed elements are exactly the ones that correspond to a closed form in $K$ under some splitting. Applications of such a differential are found in defining closed partial forms and flat partial connections, see Section~\ref{sec:Examples}. We will show that there always exists such a differential, and while the definition is not canonical, we will show that we obtain isomorphic differential complexes for any of these extra choices made.
\begin{theorem} \label{th:main}
Let $(K,d)$ be a filtered cochain complex of finite depth with filtration $K = K_0 \supset K_1 \supset K_2 \supset \cdots \supset K_S \supset K_{S+1} =0$ and $dK_j\subset K_j$ for any $j=0,\ldots,S$. 
\begin{enumerate}[\rm (a)]
\item There exists a direct sum decomposition
$$K = \calT \oplus d\calT \oplus \calP,$$
such that $(\calP,d\vert_{\mathcal{P}})$ is a subcomplex, $dK_j \subset d\calT + K_{j+1}$ for any $j = 0, \ldots, S$, and the map
$$(\calT \cap K_j)/K_{j+1} \to K_j/K_{j+1}\ , \quad \alpha \bmod K_{j+1} \mapsto d\alpha \bmod K_{j+1} \ $$
is injective. Furthermore, $(\calP, d|_{\calP})$ is isomorphic to $(K,d)$ at the cohomology level. 
Finally, if $K = \tilde \calT \oplus d\tilde \calT \oplus \tilde \calP$ is another such decomposition, then $(\calP, d|_{\calP})$ and $(\tilde \calP, d|_{\tilde \calP})$ are isomorphic as cochain complexes.
\item Let $\phi:K \to \gr(K)$ be a splitting and $K = \calT \oplus d \calT \oplus \calP$ as in \emph{(a)}. Then there exists a unique linear map $I_\phi: \calP \to \calT$ such that $d_{\gr}\phi(\alpha - I_\phi \alpha) =0$ for any $\alpha \in \calP$. Furthermore, the map
$$\alpha \in \calP \to [\phi(\alpha - I_\phi \alpha)]\in H(\gr(K))$$
is a bijective map and for any $\alpha \in \calP$, there exists a closed/exact $\beta \in K$ relative to $d$ such that $[\phi(\alpha) - I_\phi \alpha)] = [\phi(\beta)]$ if and only if $\alpha$ is closed/exact.
\end{enumerate}
\end{theorem}
From the above theorem, we can define $d_H: H(\gr(K)) \to H(\gr(K))$ by $d_H[\phi(\alpha - I_\phi \alpha)] = [\phi(d\alpha - I_\phi d\alpha)]$, identifying $H(\gr(K))$ with $\calP$, and while this differential is not canonical, we get isomorphic complexes relative to each such choice. In order to produce the subcomplex $\calP$, one needs to make certain choices related to the filtration itself, such as the operator $d_0\colon K\to K$ which we call a base differential of $(K,d)$, along with the subspaces $\calT \subset \calC$, which will be referred to as the base transversal and cotransversal of $d_0$. We will not go much into how these structures can be constructed, rather showing if that these can be found, then any such choice will give equivalent results. See some special cases of construction at the end of Section~\ref{subsection Rumin} and Section~\ref{sec:Examples}. We emphasize, however, that once the base differential, transversal and cotransveral have been constructed, the formulas are quite explicit.

Another cohomological machinery that arises when considering filtered complexes is that of spectral sequences. In this paper, we prove that a filtered cochain complex $(K,d)$ of finite depth is a multicomplex and, as such, it has further properties that allow us to describe each page of their spectral sequence explicitly. In this way, we are also able to clearly state the relationship between the subcomplexes $(\mathcal{P},d\vert_\mathcal{P})$ that we introduce and the spectral sequence differential operators and quotients. In this sense, one can think of such subcomplexes as arising ``naturally'' from the given filtration.

The structure of the paper is as follows. In Section~\ref{sec:Filtered} we give the general definition of a base differential $d_0\colon K\to K$ of a filtered complex $(K,d)$, and its relative base transversals and cotransversals $\mathcal{T}\subset\mathcal{C}$, together with some of their basic properties. We then present the explicit construction of the subcomplex $(\mathcal{P},d\vert_{\mathcal{P}})$, show how it relates to the Rumin complex, and then prove Theorem~\ref{th:main}. In Section~\ref{sec:Examples} we produce several different examples of filtered cochain complexes of finite depth, and explain how to identify their base differential and the relative base complex $(\mathcal{P},d\vert_\mathcal{P})$. In Section~\ref{sec:SpectralSeq} we focus on the construction of the spectral sequence of a given filtered complex $(K,d)$ of finite depth, given that such $(K,d)$ turns out to be a multicomplex. Finally, in Section~\ref{last section}, we present the explicit computations behind each page of the spectral sequence associated to a nonnegative grading on a 3-dimensional filtered Lie group.

\section{Subcomplexes of general filtered complexes} \label{sec:Filtered}
\subsection{Base differential of filtered complexes}
In this section, we present a general approach to finding subcomplexes of any filtered complex, which are isomorphic to the original complex at the level of cohomology. Our ideas are built on the work of graded forms presented by Rumin in \cite{Rumin3} and \cite{Rumin4}, and later used in \cite{DaHa22}. 

Let $(K,d)$ be a cochain complex with differential operator $d$. Assume that we have a sequence of sub-complexes of finite depth, that is
\begin{equation} \label{filteredK} K = K_0 \supset K_1 \supset K_2 \supset \cdots \supset K_S \supset K_{S+1} =0\text{ with }dK_j\subset K_j\,.\end{equation}
Define the graded complex as in \eqref{DefGr} with differential $d_{\gr}$.
We introduce the following definitions.
\begin{definition} \label{def:base}
We say that a linear operator $d_0:K \to K$ is \emph{a base differential} of the filtered complex $(K,d)$ if for any $j =0, \dots, S$,
\begin{enumerate}[\rm (i)]
\item $d^2_0 = 0$;
\item if $\alpha \in K_j$, then $d \alpha = d_0 \alpha \bmod K_{j+1}$;
\item $\ker d_0|_{K_j} \bmod K_{j+1} = \ker d_{\gr} |_{K_j/K_{j+1}}$.
\end{enumerate}
\end{definition}
Notice from the Definition~\ref{def:base} that $(K_j, d_0)$ is a differential complex for $j =0, 1, \dots, S$.
\begin{example} \label{ex:phiTod0}
Let $\phi: K \to \gr(K)$ be a splitting as described in Section~\ref{sec:Introduction}. We can then define corresponding base differential $d_0$ as follows. If $K = K_{(0)} \oplus \cdots \oplus K_{(S)}$ is the splitting, seen as a grading on $K$, then by definition $d = d_{(0)} + d_{(1)} + \cdots + d_{(S)}$ only has components of nonnegative degree. We can then define $d_0 = d_{(0)}$ which will satisfy all the conditions of Definition~\ref{def:base}.
\end{example}
\begin{definition}
Let $d_0:K \to K$ be a base differential of $(K,d)$.
\begin{enumerate}[\rm (i)]
\item We say that a subspace $\calT \subseteq K$ is a base transversal for $d_0$ if it satisfies for any $j= 0,1, \dots, S,$
$$K_j = \ker d_0|_{K_j} \oplus (\calT \cap K_j).$$
\item We say that a subspace $\calC \subseteq K$ is a base cotransversal for $d_0$ with transversal $\calT$ if it satisfies
$$K = \image d_0 \oplus \calC\quad\text{and} \quad \calT \subset \calC\ .$$
\end{enumerate}
\end{definition}

Given a choice for a base differential $d_0$, one can always construct a corresponding base transversal $\mathcal{T}$ and cotransversal $\mathcal{C}$.

\begin{proposition} \label{prop:existence}
Any filtered complex of finite depth $(K,d)$ admits a base differential $d_0\colon K \to K$ together with a base cotransversal and transversal $\calT \subset \calC$.
\end{proposition}

\begin{proof}
We can always construct a splitting $\phi: K \to \gr(K)$ with choices of complements, also giving us a base differential $d_0$ as in Example~\ref{ex:phiTod0}. Since $d_0 K_{(j)}  \subset K_{(j)}$, we can find a complement $\calT_{(j)}$ to $\ker d_0\cap K_{(j)}$ in $K_{(j)}$, as well as a complement $\calA_{(j)}$ to $d_0K_{(j)}$ in $\ker d_0\cap K_{(j)}$. The subspaces of $K$ given by
\begin{equation} \label{calTsplitt}\calT = \oplus_{j=0}^S \calT_{(j)}\ , \ \calC = \calT \oplus (\oplus_{j=0}^S \calA_{(j)})\ \text{ with }\ \mathcal{T}\subset\mathcal{C} \end{equation}
will be respectively a base transversal and base cotransversal to the base differential $d_0$.
\end{proof}

Once we have a base differential $d_0$ and a choice for $\calT \subset \calC$, we can define a partial inverse  of the base differential as follows.
\begin{definition}\label{d_0 inv} Let $d_0\colon K\to K$ be a base differential of $(K,d)$, and let $\mathcal{T}\subset\mathcal{C}$ be a base transversal and cotransversal for $d_0$. Then we can define the \emph{inverse} map $d_0^{-1}\colon K\to K$ of $d_0$ as follows
\begin{enumerate}[\rm (a)]
\item $\ker d^{-1}_0 = \calC$, 
\item if $\alpha \in \image d_0$, then $d_0^{-1} \alpha$ is the unique element in $\calT$ mapped to $\alpha$ by $d_0$.
\end{enumerate}
\end{definition}
By requiring $\calT \subset \calC$, it follows that $(d_0^{-1})^2 = 0$. We also need the following observation.
\begin{lemma}
For any $j = 0, \dots, S$, we have $d_0^{-1} K_j \subset K_j$.
\end{lemma}
\begin{proof}
Obviously $d_0^{-1} K_0 \subset K_0 = K$. For the remaining cases let $0 \leq i < j \leq S+1$ and  $\alpha \in (\calT \cap K_i)$ be such that $d_0 \alpha \in K_j \subset K_{i+1}$. Then $\alpha +K_{i+1} \in \ker d_{\gr}|_{K_i/K_{i+1}}$. It follows that for some $\beta \in K_{i+1}$, $\alpha +\beta \in \ker d_0|_{K_i}$. Because of the decomposition $K_{i+1} = \ker d_0|_{K_{i+1}} \oplus (\calT \cap K_{i+1})$, we can write $\beta = \beta_0 + \beta_1$ with $\beta_0 \in  \ker d_0|_{K_{i+1}}$, $\beta_1 \in(\calT \cap K_{i+1})$ which in particular means that $d_0 \beta = d_0 \beta_1$.

It then follows that $\alpha + \beta_1 \in \ker d_0|_{K_i}$. Using the decomposition $K_{i} = \ker d_0|_{K_{i}} \oplus (\calT \cap K_{i})$ and the fact that $\alpha, \beta \in (\calT \cap K_{i})$, it follows that $\alpha = -\beta \in K_{i+1}$. Iterating this argument it follows that $\alpha \in K_j$, which proves our statement.
\end{proof}
With these definitions in place, we are ready to produce the subcomplex $(\mathcal{P},d\vert_{\mathcal{P}})$ corresponding to the base differential $d_0$ and the relative base transversal and cotransversal $\mathcal{T}\subset\mathcal{C}$.

We introduce the operator
\begin{equation} \label{Poperator} P = \id - d_0^{-1} d -  dd_0^{-1}.\end{equation}
Given $S$ as in \eqref{filteredK}, we introduce the notation $P^{\infty} := P^{S}$.

\begin{definition}
    Given a base differential $d_0\colon K\to K$ of $(K,d)$ with base transversal and cotransversal $\mathcal{T}\subset\mathcal{C}$, we define $\mathcal{P}\subset K$ as
    \begin{align*}
        \mathcal{P}:=\image P^\infty\subset K\,.
    \end{align*}
    Moreover, since $dP=Pd$, we have that $(\calP,d)$ is a subcomplex of $(K,d)$ and we will call this \emph{the base subcomplex} of $d_0$.
\end{definition}
Since the construction of the operator $P$ depends not only on the base differential $d_0$, but also on the choice of base transversal and cotransversal for $d_0$, one should write $\mathcal{P}=\mathcal{P}_{d_0,\mathcal{T}\subset\mathcal{C}}$. However, in order to avoid cumbersome notation, once the map $d_0\colon K\to K$ and the subspaces $\mathcal{T}\subset\mathcal{C}$ have been fixed, we will simply write $\mathcal{P}=\image P^\infty\subset K$.

\begin{theorem} \label{th:SetUp} Let $(\mathcal{P},d)$ be the base subcomplex of $d_0\colon K\to K$ with transversal and cotransversal $\mathcal{T}\subset\mathcal{C}$.
\begin{enumerate}[\rm (a)]

\item If $\alpha\in K$ is a closed form, then $P^\infty \alpha$ is in the same cohomology class. In particular,
$$P^\infty: K \to \calP$$
is an isomorphism at the cohomology level.
\item The surjection between the two complexes 
$$K \to K/(\calT + d\calT)$$
is an isomorphism at the cohomology level. Furthermore, one can write
$$K = \calT \oplus d\calT \oplus \calP\  \  \text{with} \  \ker P^\infty = \calT \oplus d \calT,$$
so the map
$$\alpha + \calT + d\calT \longmapsto P^\infty(\alpha)$$
is a well-defined isomorphism between the complexes $\big(K/(\calT+ d\calT),d\big)$ and $(\calP,d)$.
\item One can rewrite $\mathcal{P}\subset K$ as
\begin{align*}
    \calP &= \{ \alpha \in K \, : \, P(\alpha) = \alpha \} = \ker (dd_0^{-1} + d_0^{-1} d) = \ker d_0^{-1} \cap \ker (d_0^{-1} d) = \{ \alpha \in \calC \, : \, d \alpha  \in \calC \}\ .
\end{align*}
In particular, $P^k = P^\infty$ for any $k \geq S$.  
\end{enumerate}
\end{theorem}
\begin{proof}
\begin{enumerate}[\rm (a)]
\item If $\alpha$ is closed, $P\alpha = \alpha - dd_0^{-1} \alpha$. In other words, $\alpha$ and $P\alpha$ differ by an exact form. Applying this result iteratively, we have that $\alpha$ and $P^{\infty} \alpha$ are in the same cohomology class.
\item We will prove the first part of the statement by showing that there are no closed forms in $\calT$. This observation follows from the fact that  $\calT = \image d_0^{-1}$, and so
$$d_0^{-1}d |_{\calT} =\id|_{\calT} + d_0^{-1}(d-d_0)|_{\calT},$$
which has inverse $\sum_{i=0}^{S}[(-d_0^{-1} (d-d_0)|_{\calT}]^i$. Hence, $\calT$ is transverse to $d\calT$ and $K \to K/(\calT \oplus d\calT)$ is an isomorphism at the level of cohomology. For the second part of the statement, we first observe that since $\calT = \image d_0^{-1}$,
\begin{equation} \label{Palphaminalpha} P(\alpha) = \alpha \mod \calT + d\calT.\end{equation}
Hence, if $P^k(\alpha) = 0$, then we must have $\alpha \in \calT + d \calT$, giving us that $\ker P^k \subset \calT + d\calT$ for any $k > 0$. For the inclusion in the other direction, we first observe that for $\beta \in \calT \cap K_j$, we have
$$P(\beta) = \beta - d_0^{-1} d\beta = -d_0^{-1} (d-d_0)\beta \in \calT \cap K_{j+1},$$
which also implies that $P(d\beta) = dP(\beta) \in d \calT \cap K_{j+1}$. By the finiteness of the filtration, it follows that $P^\infty(\beta) = P^\infty(d\beta) = 0$.
\item Introduce $\hat \calP = \{ \alpha \in K \, : \, P(\alpha) = \alpha \}$. If $\alpha \in \hat \calP$, then $P^\infty(\alpha) = \alpha$ by definition, so $\hat \calP \subset \calP$. Conversely, if $\alpha \in \calP$, then $ \alpha =P^\infty(\beta)$, and furthermore, since $P(\beta) - \beta$ is in $\calT + d\calT$ by \eqref{Palphaminalpha}, we have
$$P(\alpha) - \alpha = P^{\infty}(P(\beta) - \beta) = 0.$$
giving us that $\calP \subset \hat \calP$.
\end{enumerate}
\end{proof}

\begin{corollary} \label{cor:d0Decomp}
Given a base differential $d_0\colon K\to K$ of $(K,d)$ with transversal and cotransversal $\mathcal{T}\subset\mathcal{C}$, we have the following direct sum decomposition for $K$,
$$K = \calT \oplus d_0 \calT \oplus \calP = \image d_0^{-1} \oplus \image d_0 \oplus \calP.$$
Finally, $\calC = \calT \oplus \calP$.
\end{corollary}
\begin{proof}
By Theorem~\ref{th:SetUp}~(b), for any $\alpha \in K_j$, we have unique $\beta_{j}, \gamma_j \in \calT \cap K_j$, $\omega_j \in \calP \cap K_j$, such that $\alpha = \beta_j + d\gamma_j + \omega_j$, and in particular we have that
\begin{align*}
    \alpha = \beta_j + d_0 \gamma_j +(d-d_0)\gamma_j+ \omega_j \ \text{ with }\ (d-d_0) \gamma_j \in K_{j+1}\,.
\end{align*}
Repeating this process on $\alpha_{j+1} = (d-d_0)\gamma_j$, we have that there exist unique $\beta_{j+1},\gamma_{j+1}\in\mathcal{T}\cap K_{j+1}$ and $\omega_{j+1}\in\mathcal{P}\cap K_{j+1}$ such that
\begin{align*}
\alpha_{j+1} =& \beta_{j+1}+d\gamma_{j+1}+\omega_{j+1}\\
=&\beta_{j+1}+d_0\gamma_{j+1}+(d-d_0)\gamma_{j+1}+\omega_{j+1}\ ,\text{
 with 
} (d-d_0)\gamma_{j+1}\in K_{j+2}\,.
\end{align*}
Proceeding iteratively and using the fact that our filtered complex has finite length gives us the unique decomposition $\alpha = \sum_{i=j}^S \beta_i + d_0 \sum_{i=j}^S \gamma_i + \sum_{i=j}^S \omega_i$.

By Theorem~\ref{th:SetUp}~(c), we have that $\calP \subset \calC$, hence $\calT \oplus \calP \subset \calC$, and the equality follows from the fact that both $\mathcal{T}$ and $\mathcal{P}$ are transversal to $\image d_0 = d_0 \calT$.
\end{proof}

\begin{remark}
Theorem  ~\ref{th:SetUp}~(b) and Corollary ~\ref{cor:d0Decomp} imply that for any element $\alpha\in K$, there exists a unique way of expressing $\alpha$ in terms of elements $\beta,\gamma\in \mathcal{T}$ and $\omega\in\mathcal{P}$.
Using the partial inverse $d_0^{-1}\colon K\to K$, one can also obtain an explicit formulation of these elements in terms of the given $\alpha\in K$.
\begin{enumerate}[\rm (a)]
\item For $\alpha\in K$ given as $\alpha = \beta + d\gamma + \omega \in K$, with $\beta, \gamma\in \calT$, $\omega \in \calP$, then
$$\omega = P^\infty(\alpha) \qquad \gamma = \sum_{i=0}^S [-d_0^{-1}(d-d_0)]^i d^{-1}_0 \alpha, \qquad \beta = \alpha - d \gamma - \omega .$$
These identities follows from the fact that $\calC = \calT \oplus \calP$, so
$$d_0^{-1} \alpha = d_0^{-1} d \gamma = \gamma + d_0^{-1} (d-d_0) \gamma\ \Longrightarrow\ \gamma = \sum_{i=0}^S [- d_0^{-1}(d-d_0)]^i d^{-1}_0 \alpha\,.$$
\item For $\alpha\in K$ given as $\alpha = \beta + d_0\gamma + \omega \in K$, with $\beta, \gamma\in \calT$, $\omega \in \calP$, then $\gamma = d_0 d_0^{-1} \alpha$, and so
$$\omega = P^\infty(\alpha - d_0 \gamma), \qquad \beta = (\id - P^\infty)(\alpha - d_0 \gamma)\ .$$

\item Observe that by Theorem~\ref{th:SetUp}~(c), there is a unique element $\beta$ in each equivalence class $\alpha + \calT + d\calT$ satisfying $$d^{-1}_0\beta =0 \quad \text{and} \quad  d^{-1}_0 d \beta =0\ .$$
\item Moreover, the subcomplex $(\calP,d)$ is itself a filtered complex
$$\calP = \calP_0 \supset \calP_1 \supset  \cdots \supset \calP_S \supset \calP_{S+1} =0\ \text{ with }\ \calP_j = \calP \cap K_j\ .$$
\end{enumerate}
\end{remark}

\subsection{Analogues of Rumin forms and Hodge decomposition}\label{subsection Rumin}
Given a base differential $d_0$ of a filtered complex $(K,d)$ with relative base transversal and cotransversal $\mathcal{T}\subset\mathcal{C}$, let us introduce the subspace
$$\calE_0 = \ker d_0 \cap \ker d_0^{-1}  = \ker d_0 \cap \calC \subset K\ .$$
From the definition of $\mathcal{E}_0$, we readily get a direct sum decomposition of $\ker d_0$ as  $\ker d_0 = \image d_0 \oplus \calE_0$. Moreover, the map 
\begin{equation} \label{Pi0operator} \Pi_{0} = \id - d_0d_0^{-1} - d_0^{-1}d_0,\end{equation}
is a projection onto $\calE_0$ and
$$K = (\image d_0^{-1}) \oplus (\image d_0) \oplus \calE_0 = \calT \oplus d_0 \calT \oplus \calE_0\ .$$
\begin{proposition}\label{prop 2.9}
The map $\Pi_0|_{\calP}: \calP \to \calE_0$ is a bijection with inverse $P^\infty|_{\calE_0}: \calE_0 \to \calP$. It follows that, if we define $d_c: \calE_0 \to \calE_0$ as
$$d_c \alpha = \Pi_0(P^\infty(d\alpha))=\Pi_0(dP^\infty(\alpha))\ , $$
then $(\calE_0, d_c)$ and $(\calP, d|_{\calP})$ are isomorphic as differential complexes. If we introduce a filtration on $\mathcal{E}_0$ by imposing $\calE_{0,j} = \calE_0 \cap K_j$, then $(\calE_0, d_c)$ and $(\calP, d|_{\calP})$ are isomorphic as filtered complexes.
\end{proposition}
The notation $(\calE_0, d_c)$ is used to reflect Rumin's notation in \cite{Rumin4}.

\begin{proof}
We observe that if $\alpha \in \calP = \{ \alpha \in \calC \, : \, d\alpha \in \calC\}$, then
$$P^\infty(\Pi_0(\alpha)) = P^\infty(\alpha - d^{-1}_0 d_0 \alpha) = \alpha,$$
since $d_0^{-1} d_0 \alpha \in \calT$. On the other hand, if $\alpha \in \calE_0$ then
$$P^\infty(\alpha) = (\id - d_0^{-1} d)^S \alpha = \alpha \bmod \image d_0^{-1},$$
so $\Pi_0(P^\infty \alpha) = \alpha$. Observe that both $\Pi_0$ and $P^\infty$ preserve the filtration. It follows that if $d_c \alpha = \Pi_0(d P^\infty(\alpha))$ then $(\calE_0, d_c)$ and $(\calP, d|_{\calP})$ are by definition isomorphic complexes.
\end{proof}

We emphasize the following special case. Assume that we have a filtered complex $(K,d)$ as in \eqref{filteredK} with a Hilbert space structure, such that every subcomplex $K_j$ is closed in $K$. We can then use the construction in the proof of Proposition~\ref{prop:existence}, to define $K_j = K_{(j)} \oplus_\perp K_{j+1}$ where $K_{(j)}$ is the orthogonal complement of $K_{(j+1)}$ in $K_j$. We again define $d_0$ as the linear operator such that for $\alpha \in K_{(j)}$, we have $d_0\alpha$ as the projection of $d\alpha$ to $K_{(j)}$. We make the final assumption that $\ker d_0$ and $\image d_0$ are both closed subspaces. Then one can take $\calT = (\ker d_0)^\perp$ and $\calC = (\image d_0)^\perp$ as base transversal and cotransversal for $d_0$. It is also possible to consider the formal adjoint $d_0^*$ of $d_0$ with respect to the inner product of such Hilbert space. We can then write a Hodge-type decomposition
\begin{equation} \label{E0Decomp} K = \calT \oplus_\perp d_0 \calT \oplus_\perp \calE_0 = \image(d_0^*) \oplus_\perp \image(d_0) \oplus_\perp \ker(\Box_0)\ ,\text{ with }\Box_0 = d_0 d_0^* + d_0^* d_0\ ,\end{equation}
and the subcomplex $(\ker \Box_0, d_c)$ is isomorphic to $(K,d)$ at the cohomology level (see also \cite{DaHa22} and \cite{fischer2023alternative}). By contrast, in the same special case, we have no reason for the decompositions $K = \calT \oplus d\calT \oplus \calP$ and $K = \calT \oplus d_0 \calT \oplus \calP$ to be orthogonal, making \eqref{E0Decomp} preferable in some cases, see, e.g., application in \cite{GrSl25}.

\subsection{Invariance under choice of transversal and cotransversal}
Let $d_0$ be a base differential of the filtered cochain complex $(K,d)$, with base transversal and cotransversal $\calT \subset \calC$. Let $\tilde \calT \subset \tilde \calC$ be a different choice of transversal and cotransversal for the same base differential $d_0$, and define the subcomplexes $\calP =  \calP_{d_0, \calT \subset \calC}$ and $\tilde \calP =  \calP_{d_0, \tilde \calT \subset \tilde \calC}$ relative to these two different choices. The following result holds.

\begin{proposition} \label{prop:TransInv}
If $P^\infty = P^{S}$ is defined as in \eqref{Poperator} with respect to the base transversal and cotransversal $\calT \subset \calC$, while $\tilde \Pi_0$ is defined as in \eqref{Pi0operator} with respect to $\tilde \calT \subset \calC$ then $P^\infty\circ\tilde\Pi_0|_{\tilde \calP}$ is an isomorphism between $\tilde \calP$ and $\calP$ as filtered differential complexes.
\end{proposition}
\begin{proof}
Let $d_0^{-1}$ and $\tilde d_0^{-1}$ be the partial inverses of the base differential $d_0$, constructed using to $\calT \subset \calC$ and $\tilde \calT \subset \tilde \calC$ respectively. 
We introduce $\Pi_0 = \id - d_0 d_0^{-1} - d_0^{-1}d_0$ and $\tilde \Pi_0 = \id - d_0 \tilde d_0^{-1} - \tilde d_0^{-1}d_0$ as the relative projections onto $\calE_0 = \ker d_0 \cap \calC$ and $\tilde \calE_0 = \ker d_0 \cap \tilde \calC$. We first see that, for any $\alpha \in \tilde \calE_0$
$$\tilde \Pi_0(\Pi_0(\alpha)) = \tilde \Pi_0(\alpha - d_0  d_0^{-1} \alpha) = \alpha\ .$$
We have similar relations with $\Pi_0$ and $\tilde \Pi_0$ reversed. Hence, the following sequence of invertible maps
$$\tilde \calP \stackrel{\tilde \Pi_0}{\longrightarrow} \tilde \calE_0 \stackrel{\Pi_0}{\longrightarrow} \calE_0 \stackrel{P^\infty}{\longrightarrow} \calP,$$
are all linear isomorphisms that preserve the filtration induced by the filtration \eqref{filteredK} on $K$.

To conclude, it is sufficient to observe that for any $\alpha \in \tilde \calP$, we have
\begin{align*}
     P^\infty(\Pi_0( \tilde \Pi_0(\alpha))) =& P^\infty(\Pi_0( \alpha - \tilde d_0^{-1} d_0 \alpha ))=P^\infty(\alpha-\tilde d_0^{-1}d_0\alpha-d_0d_0^{-1}(\alpha-\tilde d_0^{-1}d_0\alpha)) \\
     = &P^\infty(\alpha -\tilde d_0^{-1} d_0 \alpha) = P^\infty\circ\tilde \Pi_0(\alpha)\ .
\end{align*}
\end{proof}

\subsection{Invariance under choice of base differential} 
We now look into the base differentials themselves. Let $d_0$ be a base differential for a filtered complex $(K,d)$, with base transversal and cotransversal $\calT \subset \calC$, and with $\calP = \calP_{d_0, \calT \subset \calC}$. Let $\tilde d_0$ be a different choice of base differential of the same filtered complex $(K,d)$, with similar associated objects $\tilde \calT \subset \tilde \calC$ and $\tilde \calP$. We then have the following result

\begin{proposition} \label{prop:Invariance}
\begin{enumerate}[\rm (a)]
    \item $\calT \subset \calC$ are also base transversals and cotransversals for $\tilde d_0$.
    \item If $P^\infty = P^{S}$ is defined as in \eqref{Poperator} with respect to $d_0$ and $\calT \subset \calC$ and $\tilde \Pi_0$ is defined as in \eqref{Pi0operator} with respect to $\tilde d_0$ and $\tilde \calT \subset \tilde \calC$, then $P^\infty\circ\tilde\Pi_0|_{\tilde \calP}$ is an isomorphism of differential complexes between $\tilde \calP$ and $\calP$.
\end{enumerate}
\end{proposition}
\begin{proof}
\begin{enumerate}[\rm (a)]
    \item Let us introduce the notation $\tilde d_0 = d_0 - N$, where the map $N = d_0- \tilde d_0 =  ( d-\tilde d_0) - (d- d_0)$ maps $K_j$ into $K_{j+1}$. It follows that $\hat N := d_0^{-1} N$ is nilpotent with $\hat N^{S+1} =0$, meaning that $A : =\id - \hat N$ is invertible.  If we consider the projection map $d_0d_0^{-1}$ applied to the image of $\tilde d_0$, we have that for any $\alpha\in K$
    \begin{align*}
        d_0d_0^{-1}(\tilde d_0\alpha)=d_0d_0^{-1}(d_0\alpha-N\alpha)=d_0\alpha-d_0\hat N\alpha=d_0(\alpha-\hat N\alpha)
    \end{align*}
    and so we get
    \begin{align*}
    d_0 d^{-1}_0|_{\image(\tilde d_0)} &: \image \tilde d_0 \to \image d_0\ ,\ 
    \tilde d_0 \alpha  \longmapsto d_0 A \alpha \ .
    \end{align*}
    The map given by the composition 
    $$\calT \stackrel{A^{-1}}{\longrightarrow} \calT \stackrel{\tilde d_0}{\longrightarrow} \tilde d_0 \calT \stackrel{d_0 d_0^{-1}}{\longrightarrow} \image d_0\ , $$
    simplifies into $d_0\colon\mathcal{T}\to \mathrm{image\,}d_0$, since for any $\alpha\in K$ we have
    $$d_0d_0^{-1}\tilde d_0A^{-1}\alpha=d_0(A -\id +d^{-1}_0 d_0 )A^{-1}\alpha=d_0AA^{-1}\alpha=d_0\alpha.$$
    Since $d_0$ is an invertible map from $\calT$ to $\image d_0$, this implies that $\tilde d_0|_{\calT}$ must be invertible as well. Repeating this argument for every $K_j$ and $(\calT \cap K_j)$, we have $K_j = \ker \tilde d_0|_{K_j} \oplus ( \calT \cap K_j)$. Finally, by construction we have $\mathcal{T}\subset\mathcal{C}$. Moreover the map
    $$K \to \image \tilde d_0, \qquad \alpha \mapsto \tilde d_0 A^{-1} d_0^{-1}\alpha,$$
    will be a projection onto $\image \tilde d_0$ with kernel $\ker d_0^{-1}=\calC$, implying that $\calC$ is transverse to $\image \tilde d_0$.
    \item By (a) and Proposition~\ref{prop:TransInv}, it is sufficient to prove the result in the case where $\calT \subset \calC$ equals $\tilde \calT \subset \tilde \calC$. 
    Continuing with the notation in (a), we observe that $\tilde d_0^{-1} = A^{-1} d_0^{-1}$. Furthermore, if $\tilde P = \id - \tilde d_0^{-1} d - d \tilde d_0^{-1}$, then
    \begin{align*}
    \tilde \calP & = \{ \tilde P(\alpha) = \alpha  \, : \, \alpha \in K \}
    = \{ \alpha \in K \, : \, \tilde d_0^{-1} \alpha = 0 , \tilde d_0^{-1} d\alpha = 0\} \\
    & = \{ \alpha \in K \, : \, A^{-1} d_0^{-1} \alpha = 0 , A^{-1} d_0^{-1} d\alpha = 0\} =  \calP
    \end{align*}
    since $A^{-1}$ is an invertible map of $\calT$ to itself. The result follows.
\end{enumerate}
\end{proof}

\subsection{Proof of Theorem~\ref{th:main}}
We first prove (a). The existence of a splitting $K = \calT \oplus d\calT \oplus \calP$ follows from the existence of a base differential $d_0$ of $(K,d)$ as shown in Proposition~\ref{prop:existence}. The property $dK_j \subset d\calT +K_{j+1}$ is a direct consequence of the fact that for any $\alpha \in K_j$,
\begin{align*}
d\alpha & = d_0 \alpha \bmod K_{j+1} = d_0\beta \bmod K_{j+1} = d\beta \bmod K_{j+1}\ ,\text{ where }\beta =d^{-1}_0 d_0 \alpha \in \calT\ .
\end{align*}
Conversely, assume that we have a decomposition $K = \calT \oplus d \calT \oplus \calP$ as described in Theorem~\ref{th:main}. We will show that this decomposition then also corresponds to a base differential. Let $\phi: K \to \gr(K)$ be any splitting with $d_0$ the corresponding base transversal as in Example~\ref{ex:phiTod0}.
By our assumptions on $\calT$, for any $\alpha \in K_j$ there exists a unique $\beta \in \calT$ such that
$$d \alpha = d_{0} \alpha \bmod K_{j+1} = d\beta \bmod K_{j+1} = d_{0}\beta \bmod K_{j+1}.$$
Hence, $\calT$ is a base transversal for the base differential $d_{0}$. By repeating the same reasoning of Corollary~\ref{cor:d0Decomp}, we get that $K = \calT \oplus d_{(0)}\calT \oplus \calP$, and so $\calC = \calT \oplus \calP$ gives a base cotransversal for $d_{(0)}$. 

Once we define a partial inverse map $d_{0}^{-1}\colon\mathrm{image}\,d_{0}\to\mathcal{T}$, the map $P^\infty = (\id - d_{0}^{-1} d - d d_{0}^{-1})^S\colon K\to \mathcal{P}$ gives us an isomorphism at the level of cohomology. Finally, the subcomplex $(\mathcal{P},d\vert_{\mathcal{P}})$ we have obtained in this way is isomorphic to any other base complex by Proposition~\ref{prop:Invariance}. This completes the proof of (a)

For (b), let the splitting $\phi:K \to \gr(K)$ and decompositions $K = \calT \oplus d \calT \oplus \calP$ be given. We let $d_0$ be the base differential of $\phi$ as in Example~\ref{ex:phiTod0}, and we have just shown that we can use $\calT$ as transversal and $\calC = d_0 \oplus \calP$ as cotransveral for this $d_0$. By definition of $d_0$, we have $d_{\gr} \phi = \phi d_0$, so $\phi$ is an isomorphism of differential complexes. There exists now a unique map $I_\phi: \calP \to \calT$ given by $I_\phi(\alpha) = d_0^{-1} d_0 \alpha$ such that $d_0 (\alpha - I_\phi \alpha) =0$, which is equivalent to $d_{\gr}\phi(\alpha- I_\phi \alpha) = 0$. We remark that the dependence of $I_\phi$ on $\phi$ is through the choice of base differential.

Next, we observe from the identities $d_0 \calT = \image d_0$ and $\ker d_0 = d_0 \calT \oplus \calE_0$ from Section~\ref{subsection Rumin}, it follows that each cohomology class of $d_0$ is in one-to-one correspondence between $\calE_0$. Since $\alpha - I_\phi \alpha = \Pi_0(\alpha)$, it follows that $\alpha \mapsto [\phi(\Pi_0(\alpha))]$ is an invertible map. If $\alpha \in \calP$ is closed/exact relative to $d$, then $d_0 \alpha =0$ also, so $\Pi_0(\alpha) = \alpha$ and hence the corresponding cohomology class is on the form $[\phi(\alpha)]$ for some closed/exact form. Conversely, assume that we have
$$[\phi(\Pi_0(\alpha))] = [\phi(\beta)]$$
for some $d_0$-closed $\beta \in K$ such that $\beta$ is also $d$-closed/exact. This means that there exists some $\gamma \in K$ such that
\begin{equation} \label{ClosedExact1} \alpha - d_0^{-1} d_0 \alpha = \beta + d_0 \gamma,\end{equation}
but since the left side vanishes under application of $d_0^{-1}$, so must the right hand side and we have $d_0 \gamma =- d_0d^{-1} \beta$, letting us rewrite \eqref{ClosedExact1} as
$$\Pi_0(\alpha) = \Pi_0(\beta).$$
Applying now $P^\infty$ on both sides, it follows that $\alpha = P^\infty(\beta)$ and since $P^\infty$ is a homomorphism of differential complexes, $\alpha$ is closed/exact when $\beta$ is closed/exact. This completes the proof of (b). $\Box$

We note from the last part of the proof that we get the following explicit identity
\begin{corollary}
If $\phi:K \to \gr(K)$ is any splitting and if $d_0$ is defined relative to this splitting, then for any element $[\phi(\beta)]\in H(\gr(K))$, where $\beta$ is some $d_0$-closed form, then
$$[\phi(\beta)] = [\phi(\Pi_0 P^\infty(\beta)))]$$
\end{corollary}
Hence, if we want to see if there are any closed/exact elements in the equivalence class $[\phi(\beta)]$, all we need to do is to check if $P^\infty(\beta)$ is closed/exact.

\section{Examples} \label{sec:Examples}
We consider the following examples of filtered complexes $(K,d)$ and their relative base complexes $(\mathcal{P},d)$ to illustrate our theory on some known structures.

\subsection{Manifolds with a negative filtration of the tangent bundle}\label{subsection 3.1}
Let $M$ be a connected manifold. Assume that we have a filtration on the tangent bundle
$$TM = TM_{-s} \supset TM_{-s+1} \supset \cdots \supset TM_{-2} \supset TM_{-1}\supset TM_0 =0\,,$$
where we are assuming $TM_{-s-k}=TM_{-s}=TM$ for any $k\ge 0$.
This gives us a filtration
$$\bigwedge TM  = (\bigwedge TM)_{-S} \supset (\bigwedge TM)_{-S+1} \supset \cdots \supset (\bigwedge TM)_{-1} \supset (\bigwedge TM)_{0}\supset 0\ ,$$
\begin{align}\label{def homogeneous dimension}
    \text{where }S = \sum_{j=1}^s j \left(\rank (TM)_{-j} - \rank (TM)_{-j+1}\right)\,,
\end{align}
defined by imposing $\chi_1 \wedge \chi_2\in  (\bigwedge TM)_{-i-j}$, if $\chi_1 \in (\bigwedge TM)_{-i}$ and $\chi_2 \in (\bigwedge TM)_{-j}$. We remark that from this convention $M \times \mathbb{R} = \bigwedge^0 TM$ is in $(\bigwedge TM)_{0}$, but not in $(\bigwedge TM)_{-1}$. 

By duality, we introduce $(\bigwedge T^\ast M)_{j}$ defined as the subspace of the elements in $\bigwedge T^*M$ that vanish on $(\bigwedge TM)_{-j+1}$, for any $j\ge 1$. We will denote by $\Omega_j$ the space of all smooth sections of $(\bigwedge T^\ast M)_{j}$. Just like in the case of the filtration on $\bigwedge TM$, $\Omega_0$ is the space of all smooth forms, so that $0$-forms $\Omega^0=C^\infty(M)$ will belong to $\Omega_0$ but not $\Omega_1$.

It is worth mentioning that, by definition, the filtration on $\bigwedge TM$ gives rise to a dual decreasing filtration on $\bigwedge T^*M$
\begin{align*}
    \bigwedge T^*M=(\bigwedge T^*M)_0\supset(\bigwedge T^*M)_1\supset\cdots\supset (\bigwedge T^*M)_S\supset(\bigwedge T^*M)_{S+1}=0\,,
\end{align*}
where $S$ is the index given in \eqref{def homogeneous dimension}. This filtration then naturally extends to the space of smooth forms as
\begin{align*}
\Omega=\Omega_0\supset\Omega_1\supset\cdots\supset\Omega_{S}\supset\Omega_{S+1}=0\,.
\end{align*}
 
We introduce the following crucial assumption on our filtration: If $\Gamma(TM_{-j})$ denotes the the sections of $TM_{-j}$, then for any $i,j=0,1,\dots, s$
\begin{equation}
    \label{A} \tag{A} \text{if } X\in \Gamma(TM_{-i}) \text{ and } Y \in \Gamma(TM_{-j}) \, \text{ then }[X,Y] \in \Gamma(TM_{-i-j}).
\end{equation}
\begin{lemma} \label{lemma:DifferentialIdeal}
If \eqref{A} holds, then for $j =0,1, \dots, S$, we have $d\Omega_j \subset \Omega_j$.
\end{lemma}
\begin{proof}For $j=0$, the claim is trivial. Let $\alpha\in\Omega_j$ be a $k$-form with $j\ge 1$ and prove that also $d\alpha\in\Omega_j$.
For any $k+1$ smooth vector fields $X_1,\ldots,X_{k+1}$ on $M$, the formula for the exterior differential on $\alpha$ is given by
\begin{equation}\label{formula exterior derivative}
\begin{aligned}
    d\alpha (X_1\wedge\cdots\wedge X_{k+1})=&\sum_{1\le i\le k+1}(-1)^{i-1}X_i\big(\alpha(X_1\wedge\cdots\wedge\hat{X}_i\wedge\cdots\wedge X_{k+1})\big)+\\+&\sum_{1\le i<l\le k+1}(-1)^{i+l}\alpha\big([X_i,X_l]\wedge X_1\wedge\cdots\wedge\hat{X}_i\wedge\cdots\wedge\hat{X}_l\wedge\cdots\wedge X_{k+1}\big)\,,
\end{aligned}
\end{equation}
where the hat-symbol is used to indicate the terms omitted. If $X_1 \wedge \cdots \wedge X_{k+1}$ is a section of $(\bigwedge TM)_{-j+1}$, then so is $X_1 \wedge \cdots \wedge \hat X_i \wedge \cdots \wedge X_{k+1}$ for $i=1, \dots, k+1$, and by assumption \eqref{A} this is the case for $[X_i, X_l] \wedge X_1 \wedge \cdot \wedge \hat X_i \wedge \cdots \wedge \hat X_i \wedge \cdots \wedge X_l \wedge \cdots  \wedge X_{k+1}$, $i,l=1, \dots, s$. We hence have that both terms vanish in \eqref{formula exterior derivative}.
\end{proof}

Given the decreasing filtration over the space of smooth forms, it is possible to express $\Omega$ as a direct sum decomposition $\Omega = \Omega_{(0)} \oplus \cdots \oplus \Omega_{S}$, where the subspaces $\Omega_{(j)}$ are taken so that $\Omega_j = \Omega_{(j)} \oplus \Omega_{j+1}$. We observe that if $\alpha \in \Omega_j$, then
$$d(f\alpha) = f d\alpha  \mod\Omega_{j+1}.$$
Hence, we can define a $C^\infty(M)$-linear map $d_0\colon\Omega_j\to\Omega_j$ such that, for any $\alpha \in\Omega_{(j)}$, 
$$d_0\alpha:=d\alpha\mod\Omega_{j+1}\ \text{ so that } \ d_0\alpha \in\Omega_{(j)}\,, \text{ and }\ (d-d_0) \alpha  \in \Omega_{j+1}\ .$$
By tensoriality, the map $d_0\colon\Omega\to\Omega$ induces a vector bundle morphism $\check{d}_0\colon\bigwedge T^*M\to\bigwedge T^*M$ for which $\check{d}_0^2=0$. We can then construct subbundles $\check{\mathcal{T}}\subset\check{\mathcal{C}}\subset\bigwedge T^*M$ so that
\begin{align*}
    \bigwedge T^*M=\ker\check{d}_0\oplus\check{\mathcal{T}}=\mathrm{image\,}\check{d}_0\oplus\check{\mathcal{C}}.
\end{align*}
It is therefore sufficient to take the spaces of smooth sections $\mathcal{T}\subset\mathcal{C}$ of $\check{\mathcal{T}}\subset\check{\mathcal{C}}$ to obtain a base transversal and cotransversal for $d_0\colon\Omega\to\Omega$. Indeed, the map $d_0\colon\Omega\to\Omega$ is a base differential for the de Rham complex $(\Omega,d)$, since for any $\alpha\in \Omega_j$, we have $(d-d_0)\alpha \in\Omega_{j+1}$. 

These definitions lead to well-defined base complexes $(\mathcal{P},d\vert_\mathcal{P})$ as in Theorem \ref{th:main}, and $(\mathcal{E}_0,d_c)$ as in Section \ref{subsection Rumin}. In particular, the latter subcomplex coincides with the complex defined by Rumin in \cite{Rumin3} and \cite{Rumin4}, where
\begin{enumerate}[$\bullet$]
\item grading of the tangent bundle is defined so that $TM_{-j-1}$ is spanned by $TM_{-j}$ brackets of vector fields with one taking value in $TM_{-j}$ and the other in $TM_{-1}$;
\item the subspaces $\mathcal{T}$ and $\mathcal{C}$ are defined by taking the orthogonal complements with respect to an auxiliary Riemannian metric on $M$ of $\ker d_0$ and $\mathrm{image\,}d_0$ respectively.
\end{enumerate}

\subsubsection{Application: Differentiating partial forms} \label{sec:Application} Assume for the sake of simplicity that $M$ is simply connected. Let $\eta \in \Gamma(TM_{-1}^*)$ be a one-form defined only on the subbundle $TM_{-1}$. Even though we do not have a well-defined differential, we have an idea of what closed should mean for such forms; the integral of $\eta$ along any curve should only depend on its endpoints. We can only integrate along curves that are tangent to $TM_{-1}$, but if $TM_{-1}$ satisfies the bracket-generating condition, meaning that vector fields that takes values in it and their iterated brackets span the tangent bundle, all curves tangent to $TM_{-1}$ can reach any point of the manifold \cite{Ras38,Cho39}. Hence, for the latter case, we are interested in when $\eta \in \Gamma(TM_{-1}^*)$ is a closed form. Identify $\Gamma(TM_{-1}^*)$ with $\Omega^1_1/\Omega^1_2$ and consider this as a subspace of $\gr(\Omega)^1$, where we have extended by zero in the other components. Furthermore, we can identify $(\ker d_0)^1$ with $H^1(\gr(\Omega))$ since $d_0|_{\Omega^0}$ =0. Let $\phi:\Omega \to \gr(\Omega)$ be any splitting. The first test is then to check if $d_{\gr} \eta =0$ where $d_{\gr}$ is defined with $\phi$, as this is a necessary condition. Given that this holds, it is now sufficient to check that
$$dP^\infty\phi^{-1} \eta =0$$
to show that integrals along $\eta$ are indeed independent of paths.

\begin{example}
For $\kappa \in \mathbb{R}$, let $G_\kappa$, be the simply connected Lie group with Lie algebra $\frakg_\kappa = \spn_{\mathbb{R}} \{ X, Y, Z \}$ with brackets
$$[X,Y] = Z, \qquad [Y,Z] = \kappa X, \qquad [X,Z] = -\kappa Y.$$
We will also use the symbols $X, Y, Z$ for the corresponding left invariant vector fields, with a dual frame of forms $X^*$. $Y^*$, $Z^*$. We define $TM_{-1} = \spn \{ X,Y\}$ and $TM_{-2} = TM$. Let $\eta = (\eta^1 X^* + \eta^2 Y^*)|_{TM_{-1}}$ be a partial one-form. We want to check if $\eta$ can be extended to a closed one-form. We can define $d_0: \Omega \to \Omega$ as a tensorial map respecting the wedge product and satisfying
$$d_0 X^* = 0, \qquad d_0 Y^* = 0, \qquad d_0Z^* = - X^* \wedge Y^*.$$
In other words, $d_0 = Y^* \wedge X^* \wedge \iota_Z$, where $\iota_Z$ denotes the contraction. If we choose transversal $\calT = \spn_{C^\infty} Z^*$ and cotransversal $\calC = \spn_{C^\infty} \{ X^*, Y^*, Z^*, X^* \wedge Z^*, Y^* \wedge Z^*, X^* \wedge Y^* \wedge Z^*\}$, then $d_0^{-1} = Z^* \wedge \iota_X \iota_Y$. This gives us that the partial one-form is the restriction of a closed form if and only if
$$P^\infty(\eta) = P(\eta) = (\id - d^{-1}_0 d - d d_0^{-1})(\eta^1 X^* + \eta^2 Y^*) = \eta^1 X^* + \eta^2 Y^* - (Y\eta^1 - X\eta^2) Z^* $$
is closed. We will study forms on these spaces more in Section~\ref{last section}.
\end{example}

\subsection{Graded Lie algebras}\label{sec:GradedLie}
Let $\frakg = \frakg_{-s} \oplus \cdots \oplus \frakg_{-1} \oplus \frakg_{0}$ be a non-positively graded Lie algebra with the Lie bracket is a degree zero operator. This grading induces a non-negative grading on $\frakg^*$ and then furthermore on the exterial algebra $\wedge \frakg$. Let $\rho$ be a representation of $\frakg$ on a vector space $V$. We introduce a complex $K = \oplus_{k=0}^{n} K^k$, $K^k = \wedge^k \mathfrak{g}^* \otimes V$ with filtration
$$K = K_0 \supset K_1 \supset \cdots \supset K_S \supset K_{S+1} =0, \qquad K_j = ( \wedge \frakg^* )_{i \geq j} \otimes V.$$
Introduce the Lie algebra cohomology differential corresponding $d_\rho$ on $K$ corresponding to $\rho$, which is defined by
\begin{align*}
   d_\rho \alpha(X_1, \dots, X_{k+1}) & = \sum_{1\le i\le k+1} (-1)^{i-1} \rho(X_i) \alpha(X_1, \dots, \hat X_i, \dots, X_{k+1}) \\
    & \qquad + \sum_{1\le i<l<k+1} (-1)^{i+l} \alpha([X_i,X_l], X_1, \dots, \hat X_i, \dots, \hat X_l, \dots, X_{k+1})\ .
\end{align*}
with $X_1, \dots, X_{k+1} \in \frakg$. Let $d_0$ be defined in a similar way, but with respect to the trivial representation on $V$. Then $d^2_\rho =0$ and $d_0^2 =0$ and $d_0$ is a base differential of $d$. To define transversal and cotransversal, we can for example define an inner product on $K$ and define $\calT = (\ker d_0)^\perp$ and $\calC = (\ker d_0)$. This complex is used for finding unique ways to extend partial Cartan connections in full Cartan connections in \cite{GrSl25}.

\section{Spectral methods for the filtered subcomplex} \label{sec:SpectralSeq}
In many of the examples considered in Section~\ref{sec:Examples}, not only do we have a finite depth filtration of our complex $(K,d)$ as in \eqref{filteredK}, but the complex is also given an additional grading $(K^k)$, $k=0$,$1$, $\dots$. 
This is indeed the case for our examples in Section~\ref{sec:Examples}, where an element belongs to $K^k$ if it is a $k$-form. We observe that the grading satisfies 
\begin{equation} \label{GradingCond}
dK^k \subseteq K^{k+1}, \qquad K^k \subset K_k, \qquad k \in \mathbb{N}_0\ .
\end{equation}
A very natural question arises about the spectral sequences that can be constructed from such a filtration. 

It is possible to endow the complex $(K,d)$ with a bicomplex structure by considering $d=d_0+(d-d_0)$, since it satisfies the relations $d_0^2=d_0(d-d_0)+(d-d_0)d_0=(d-d_0)^2=0$, and consequently study its associated spectral sequence. This construction mimics more closely the approach used in the previous section to extract the base subcomplex $(\mathcal{P},d)$. However, we will see that, just like in the simpler case of Carnot groups \cite{lerario2023multicomplexes}, any filtered complex of finite depth is also a multicomplex, a generalisation of the notion of a (graded) chain complex and that of a bicomplex. As shown in \cite{livernet2020spectral}, multicomplexes carry extra algebraic structure that allows us to obtain an explicit formulation of the differentials that arise from the spectral sequence obtained from this finite filtration. In this way, we are also able to compare these differentials with the projection map $P^\infty$ introduced in Section~\ref{sec:Filtered}.

\subsection{Spectral sequence associated to a multicomplex}
We will first include the basic definitions that we will use for the remainder of the section.
\begin{definition}\label{def: multicomplex} A multicomplex, also called a twisted (co)chain complex, is a $(\mathbb Z,\mathbb Z)$-graded $k$-module $C$ equipped with maps $d_i\colon C\to C$ for $i\ge 0$ of bidegree $\vert d_i\vert=(i,1-i)$ such that
\begin{align}
    \sum_{i+j=l}d_id_j=0\ \text{ for all }\ l\ge 0\,.
\end{align}
    For $C$ a multicomplex and $(a,b)\in\mathbb Z\times\mathbb Z$, we write $C_{a,b}$ for the $k$-module in bidegree $(a,b)$.
\end{definition}
Given a multicomplex $C$, one could consider different possible total complexes, however we will focus on the following one.
\begin{definition}
    For a multicomplex $C$, its associated total complex $\mathrm{Tot}\,C$ is the chain complex with
    \begin{align*}
    \left(\mathrm{Tot}\,C\right)_k=\bigg(\prod_{\substack{a+b=k\\ a\le 0}}C_{a,b}\bigg)\oplus\bigg(\bigoplus_{\substack{a+b=k\\ a>0}}C_{a,b}\bigg)=\bigg(\bigoplus_{\substack{a+b=k\\ b\le 0}}C_{a,b}\bigg)\oplus\bigg(\prod_{\substack{a+b=k\\ b>0}}C_{a,b}\bigg)\,.
    \end{align*}
    The differential on $\mathrm{Tot}\,C$ is given for an arbitrary element $c\in\big( \mathrm{Tot\,C}\big)_k$ by
    \begin{align}\label{formula 5}
        (dc)_a=\sum_{i\ge 0}d_i(c)_{a-i}\,,
    \end{align}
    where $(c)_a$ denotes the projection of $c\in\big(\mathrm{Tot}\,C\big)_k$ to $C_{a,\bullet}=\prod_{b,a+b=k}C_{a,b}$.
\end{definition}
In general, when working with $(\mathrm{Tot}\,C)_k$ it is not always possible to consider the direct product total complex $\prod_{a+b=k}C_{a,b}$, as the formula \eqref{formula 5} may involve infinite sums. However, since in this paper we are working on a complex of finite depth, the associated total complex takes the simpler form
\begin{align}\label{formula 6}
    \big(\mathrm{Tot}\,C\big)_k=\bigoplus_{\substack{a+b=k\\ a\ge 0}}C_{a,b}\ , \ \text{ with differential }(dc)_a=\sum_{i=0}^Sd_i(c)_{a-i}\,.
\end{align}
\begin{definition}
    Given a multicomplex $C$ and its associated total complex $\mathrm{Tot}\,C$, we can define for each degree $k$ the following subcomplexes
    \begin{align*}
        \big(\mathcal{F}_p\mathrm{Tot\,}C\big)_k=\bigoplus_{\substack{a+b=k\\ a\ge p}}C_{a,b}\,.
    \end{align*}
    By definition, as the value of $p$ varies, the subcomplexes $\mathcal{F}_p\mathrm{Tot}\,C$ form a filtration of $\mathrm{Tot}\,C$ of finite depth:
    \begin{align*}
        \mathrm{Tot}\,C=\mathcal{F}_0\mathrm{Tot}\,C\supset\mathcal{F}_1\mathrm{Tot}\,C\supset\cdots\supset\mathcal{F}_S\mathrm{Tot}\,C\supset\mathcal{F}_{S+1}\mathrm{Tot}\,C=0\,.
    \end{align*}
\end{definition}
Let us now consider the spectral sequence associated with this filtered complex $\mathcal{F}_p\mathrm{Tot}\,C$. For $r\ge 0$, the $r^{th}$-page of the spectral sequence is a bigraded module $E_r^{p,\bullet}(\mathrm{Tot}\,C)$ with a map $\delta_r$ of bidegree $(r,1-r)$ for which $\delta_r\circ \delta_r=0$. Moreover, we have that the spaces $E_r^{p,\bullet}(\mathrm{Tot}\,C)$ can be expressed as the quotients
\begin{align*}
    E_r^{p,\bullet}(\mathrm{Tot}\,C)\cong\mathcal{Z}_r^{p,\bullet}(\mathrm{Tot}\,C)/\mathcal{B}^{p,\bullet}_r(\mathrm{Tot}\,C)
\end{align*}
where the $r$-cycles are given by
\begin{align*}
    \mathcal{Z}_r^{p,\bullet}(\mathrm{Tot}\,C):= \left\{ c \in \mathcal{F}_p\mathrm{Tot}\,C \, : \, d\alpha \in \mathcal{F}_{p+r}\mathrm{Tot}\,C \right\} 
\end{align*}
and the $r$-boundaries are given by
\begin{align*}
    \begin{cases}
        \mathcal{B}_0^{p,\bullet}(\mathrm{Tot}\,C):=\mathcal{Z}_0^{p+1,\bullet}(\mathrm{Tot}\,C)\ \text{
         and 
        }\\ \mathcal{B}_r^{p,\bullet}(\mathrm{Tot}\,C):=\mathcal{Z}_{r-1}^{p+1,\bullet}(\mathrm{Tot}\,C)+d\mathcal{Z}_{r-1}^{p-(r-1),\bullet}(\mathrm{Tot}\,C)\ \text{
         for 
        }r\ge 1\,.
    \end{cases}
\end{align*}
Given an element $z\in\mathcal{Z}_r^{p,\bullet}(\mathrm{Tot}\,C)$, we will denote by $[x]_r$ its image in $E^{p,\bullet}_r(\mathrm{Tot}\,C)$, so that
\begin{align*}
    \delta_r([x]_r)=[dx]_r\ ,\ \text{
     for any 
    }[x]_r\in E^{p,\bullet}_r(\mathrm{Tot}\,C)\,.
\end{align*}
Expanding on the expressions for $\mathcal{Z}_r^{p,\bullet}(\mathrm{Tot}\,C)$ and $\mathcal{B}_r^{p,\bullet}(\mathrm{Tot}\,C)$, one can introduce the following definition.
\begin{definition}Let $x\in C_{p,\bullet}$ and let $r\ge 1$. We define the graded submodules $Z_r^{p,\bullet}$ and $B^{p,\bullet}_r$ of $C_{p,\bullet}$ as follows:
\begin{align}
    x\in Z_r^{p,\bullet}\ \Longleftrightarrow & \begin{array}{c} \text{
     for 
    }1\le j\le r-1\ \text{ there exists } z_{p+j}\in C_{p+j,\bullet}\ \text{ such that } \\ d_0x=0\ \text{
     and 
    }\ d_mx=\sum_{i=0}^{m-1}d_iz_{p+m-i}\ \text{ for all }\ 1\le m\le r-1
    \end{array}
    \label{4,4}
    \\ \nonumber \\
    x\in B_r^{p,\bullet}\ \Longleftrightarrow & \begin{array}{c} \text{ for }1\le k\le r-1 \ \text{ there exists }c_{p-k}\in C_{p-k,\bullet} \ \text{ such that} \\ 
        x=\sum_{k=0}^{r-1}d_kc_{p-k}\ \text{
         and 
        }\\0=\sum_{k=l}^{r-1}d_{k-l}c_{p-k}\ \text{
         for 
        }\ 1\le l\le r-1\,.
    \end{array}
\end{align}
    
\end{definition}
As shown already in \cite{livernet2020spectral} and \cite{lerario2023multicomplexes}, the following holds.
\begin{proposition}
    The map
    \begin{align*}
        \psi\colon\mathcal{Z}_r^{p,\bullet}(\mathrm{Tot\,}C)/\mathcal{B}_r^{p,\bullet}(\mathrm{Tot}\,C)\longrightarrow Z_r^{p,\bullet}/B_r^{p,\bullet}
    \end{align*}
    sending $[x]_r$ to the class $[(x)_p]$ is a well-defined isomorphism. Moreover, under this isomorphism, the $r^{th}$-differential of the spectral sequence corresponds to the following map
    \begin{align*}
        \partial_r\colon Z_r^{p,\bullet}/B_r^{p,\bullet}\to Z_r^{p+r,\bullet}/B_r^{p+r,\bullet}\ ,\ \partial_r\big([x]\big)=\bigg[d_rx-\sum_{i=1}^{r-1}d_iz_{p+r-i}\bigg]
    \end{align*}
    where $x\in Z_r^{p,\bullet}$ and the elements $z_{p+j}\in C_{p+j,\bullet}$ satisfy \eqref{4,4}.
\end{proposition}
\subsection{Filtered complexes $(K,d)$ of finite depth as multicomplexes}
We apply the formalism of multicomplexes to a differential complex $(K,d)$ having both a filtration 
\begin{equation} \label{grading, direct sum decomposition} K_0 \supset K_1 \supset \cdots \supset K_S \supset K_{S+1} =0 \end{equation}
and a grading $K = K^0 \oplus \cdots \oplus K^n$ satisfying \eqref{GradingCond}.
\begin{proposition}
    Any filtered complex $(K,d)$ as described above is a multicomplex with maps $d_{(i)}\colon K\to K$ of bidegree $\vert d_{(i)}\vert=(i,1-i)$ with $i=0,\ldots,S$.
\end{proposition}
\begin{proof}
    As already pointed out in \eqref{grading, direct sum decomposition}, we can express the complex $K$ as a direct sum decomposition on each degree. For each degree $k\in\mathbb N$, we have that
    \begin{align*}
        K^k=K_0^k=K_{(0)}^k\oplus K_{1}^k=K_{(0)}^k\oplus K^k_{(1)}\oplus K_2^k=K_{(0)}^k\oplus K^k_{(1)}\oplus\cdots\oplus K_{(S-1)}^k\oplus K^k_S\ .
    \end{align*}
   This direct sum decomposition of $K$ based on the filtration induces a decomposition of the original differential $d$, since $dK_j^k\subset K_j^{k+1}$. Indeed
   \begin{align*}
       d=d_{(0)}+d_{(1)}+\cdots+d_{(S)}\ \text{ where }\ d_{(i)}K_{(j)}^k\subset K_{(j+i)}^{k+1}\ ,\  i,j=0,\ldots,S\,.
   \end{align*}
   Finally, since $(K,d)$ is a complex, once we expand the formula $d^2=0$ and gather all the terms according to their respective subspace $K_{(j)}$ that they belong to, we get that for any $\alpha\in K_{(j)}^k$
   \begin{align*}
0=&d^2\alpha=\big(d_{(0)}+d_{(1)}+\cdots+d_{(S)}\big)\big(d_{(0)}+d_{(1)}+\cdots+d_{(S)}\big)\alpha\\=&\underbrace{d_{(0)}^2\alpha}_{\in K_{(j)}^{k+2}}+\underbrace{d_{(0)}d_{(1)}\alpha+d_{(1)}d_{(0)}\alpha}_{\in K_{(j+1)}^{k+2}}+\underbrace{d_{(0)}d_{(2)}\alpha+d_{(1)}^2\alpha+d_{(2)}d_{(0)}\alpha}_{\in K_{(j+2)}^{k+2}}+\cdots+\underbrace{d_{(S)}^2\alpha}_{\in K_{(j+2S)}^{k+2}}=\sum_{l=1}^{2S}\sum_{i+j=l}d_{(i)}d_{(j)}\alpha\,.
   \end{align*}
   We then have that each addend $\sum_{i+j=l}d_{(i)}d_{(j)}\alpha$ belongs to a different subspace $K_{(j+l)}^{k+2}$ and so, by the direct sum decomposition $K_j^{k+2}=\bigoplus_{l=0}^SK_{(j+l)}^{k+2}$, we get that each addend will necessarily be zero.
\end{proof}
\begin{corollary}
    Any filtered complex $(K,d)$ of finite depth as in \eqref{grading, direct sum decomposition} admits a spectral sequence associated to the specific filtration considered, whose differentials at each step $r$ have the following expression
    \begin{align*}
        \partial_r\colon Z_r^{p,\bullet}/B_r^{p,\bullet}\to Z_r^{p+r,\bullet}/B_r^{p+r,\bullet}\ ,\ \partial_r\big([x]\big)=\bigg[d_{(r)}x-\sum_{i=1}^{r-1}d_{(i)}z_{p+r-i}\bigg]\ ,\ r=1,\ldots,S\,.
    \end{align*} 
\end{corollary}
We stress that, given a filtered complex $(K,d)$ of finite depth with base differential $d_0=d_{(0)}\colon K\to K$, once we fix a base transversal and cotransversal $\mathcal{T}\subset\mathcal{C}$, it is possible to define the partial inverse $d_0^{-1}\colon K\to K$ of $d_0$ as in Definition \ref{d_0 inv}. Therefore, using this additional map, it is possible to rephrase the condition $x\in Z_r^{p,\bullet}$ of \eqref{4,4} in terms of this operator $d_0^{-1}$, and require the $z_{p+j}\in K_{(p+j)}$ to be elements in $\mathcal{T}$.

In other words, given $x\in Z_r^{p,\bullet}$, we have that there exist $z_{p+j}\in K_{(p+j)}\cap\mathcal{T}$ with $1\le j\le r-1$ such that
\begin{align*}
    z_{p+1}=&d_0^{-1}d_{(1)}x\\ z_{p+2}=&d_0^{-1}(d_{(2)}x-d_{(1)}z_{p+1})=d_0^{-1}d_{(2)}x-d_0^{-1}d_{(1)}d_0^{-1}d_{(1)}x\\ z_{p+3}=&d_0^{(-1)}(d_{(3)}x-d_{(2)}z_{p+1}-d_{(1)}z_{p+2})\\=&d_0^{-1}d_{(3)}x-d_0^{-1}d_{(2)}d_0^{-1}d_{(1)}x-d_0^{-1}d_{(1)}d_0^{-1}d_{(2)}x+d_{0}^{-1}d_{(1)}d_0^{-1}d_{(1)}d_0^{-1}d_{(1)}x\,.
\end{align*}
If we introduce the multi-index notation
\begin{align*}
    (d_0^{-1}d)_{I_m^j}:=(d_0^{-1}d_{(i_1)})(d_0^{-1}d_{(i_2)})\cdots(d_0^{-1}d_{(i_m)})
\end{align*}
where $I_m^j=(i_1,\ldots,i_m)\in\mathbb N_+^m$ for which $\vert I_m^j\vert=i_1+i_2+\cdots+i_m=j$, then for any $j=1,\ldots,r-1$, we have the expression
\begin{align}\label{18}
    z_{p+j}=\sum_{m=1}^j(-1)^{m-1}\sum_{I_m^j}(d_0^{-1}d)_{I_m^j}x\,.
\end{align}

\begin{remark}
Applying the definition of the subspaces $\mathcal{E}_0$ and Proposition \ref{prop 2.9}, there is a clear correspondence between the subspaces $\mathcal{P}=\mathcal{P}_{d_0,\mathcal{T}\subset\mathcal{C}}$ and the first page quotients of the spectral sequence, as
\begin{align*}
    \mathcal{P}=\mathcal{P}_{d_0,\mathcal{T}\subset\mathcal{C}}\cong \oplus_{p=0}^S E^{p,\bullet}_1\,.
\end{align*}

    Moreover, if we use the explicit expression of the $z_{p+j}\in K_{(p+j)}\cap\mathcal{T}$ in terms of $x$ given in \eqref{18}, it is possible to obtain the following expression for the $r^{th}$-differential of the spectral sequence
    \begin{align*}
        \partial_r([x])=\bigg[d_{(r)}x-\sum_{i=1}^{r-1}d_{(i)}\bigg(\sum_{m=1}^{r-i}(-1)^{m-1}\sum_{I_m^{r-i}}(d_0^{-1}d)_{I_m^{r-i}}\bigg)x\bigg]
    \end{align*}
    It is then a straightforward computation to check that the differential operator $dP^\infty$ introduced in section \ref{sec:Filtered} applied to any element $x\in Z_1^{p,\bullet}$ coincides with the sum of the differentials $\partial_r$ that appear in the multicomplex spectral sequence
generated by the filtration \eqref{grading, direct sum decomposition}. This result is analogous to the main result contained in \cite{lerario2023multicomplexes}, where considerations were only limited to the de Rham complex on Carnot groups.
\end{remark}

\section{Three dimensional Lie groups with negative grading}\label{last section}
We will include the following examples of the theory of Section~\ref{sec:SpectralSeq} which also cover non-nilpotent Lie groups.
Let us consider the 3-dimensional simply connected Lie group $G_\kappa$, whose Lie algebra $\mathfrak{g}_\kappa=\spn_{\mathbb R}\lbrace X,Y,Z\rbrace$,which can be identified with the $\mathbb{R}$-module of left-invariant vector fields, has the following nontrivial bracket relations
\begin{align*}
    [X,Y]=Z\ ,\ [Y,Z]=\kappa X\ ,\ [X,Z]=-\kappa Y\,.
\end{align*}
Let us stress that $G_1$ is isomorphic to $\SU(2)$, $G_0$ coincides with the first Heisenberg group, while $G_{-1}$ is the universal cover group of $\SL(2)$.

In this example, we will be showing the construction of the spectral sequence associated with the following ``sub-Riemannian'' grading:
\begin{align*}
    \mathfrak{g}_\kappa=\mathfrak{g}_{-2}\oplus\mathfrak{g}_{-1}\ , \text{ where } \mathfrak{g}_{-1}=\langle X,Y\rangle\text{ and }\mathfrak{g}_{-2}=\langle Z\rangle.
\end{align*}
Let us denote by $\lbrace X^\ast,Y^\ast,Z^\ast\rbrace$ the dual basis of $\lbrace X,Y,Z\rbrace$ on $\mathfrak{g}_\kappa^\ast$.

Using this grading and following the same construction described in subsection \ref{sec:GradedLie}, we obtain a decreasing and finite filtration over the space $\Omega^\bullet(G_\kappa)$ of smooth forms:
\begin{align*}
    \mathcal{F}_0\Omega^\bullet(G_\kappa)=\Omega^\bullet(G_\kappa)\supset\mathcal{F}_1\Omega^\bullet(G_\kappa)\supset\mathcal{F}_2\Omega^\bullet(G_\kappa)\supset\mathcal{F}_3\Omega^\bullet(G_\kappa)\supset\mathcal{F}_4\Omega^\bullet(G_\kappa)\supset\mathcal{F}_5\Omega^\bullet(G_\kappa)=\lbrace 0\rbrace
\end{align*}
where
\begin{itemize}
    \item $\mathcal{F}_0\Omega^\bullet(G_\kappa)=\Omega^\bullet(G_\kappa)$;
    \item $\mathcal{F}_1\Omega^\bullet(G_\kappa)= \spn_{C^\infty(G_\kappa)}\lbrace X^\ast,Y^\ast,Z^\ast,X^\ast\wedge Y^\ast, X^\ast\wedge Z^\ast, Y^\ast\wedge Z^\ast, X^\ast\wedge Y^\ast\wedge Z^\ast\rbrace$;
    \item $\mathcal{F}_2\Omega^\bullet(G_\kappa)= \spn_{C^\infty(G_\kappa)}\lbrace Z^\ast,X^\ast\wedge Y^\ast, X^\ast\wedge Z^\ast, Y^\ast\wedge Z^\ast, X^\ast\wedge Y^\ast\wedge Z^\ast\rbrace$;
    \item $\mathcal{F}_3\Omega^\bullet(G_\kappa)= \spn_{C^\infty(G_\kappa)}\lbrace  X^\ast\wedge Z^\ast, Y^\ast\wedge Z^\ast, X^\ast\wedge Y^\ast\wedge Z^\ast\rbrace$;
    \item $\mathcal{F}_4\Omega^\bullet(G_\kappa)= \spn_{C^\infty(G_\kappa)}\lbrace  X^\ast\wedge Y^\ast\wedge Z^\ast\rbrace$;
    \item $\mathcal{F}_5\Omega^\bullet(G_\kappa)=\lbrace 0\rbrace$.
\end{itemize}

Following the notation adopted in Section~\ref{sec:GradedLie}, we are considering the de Rham complex $(\Omega^\bullet(G_\kappa),d)$ as our filtered complex $(K,\tilde d)$, where $\tilde d$ is the exterior derivative and it admits the splitting $d=d_0+d_{(1)}+d_{(2)}$. This can be seen explicitly by writing down the action of $d$ on an arbitrary 1-form $\alpha=fX^\ast+gY^\ast+hZ^\ast\in\Omega^1(G_\kappa)=K^1$ with $f,g,h\in C^\infty(G_\kappa)=K^0$:
\begin{align*}
    d \big(fX^\ast+gY^\ast+hZ^\ast\big)=&df\wedge X^\ast+f\,dX^\ast+dg\wedge Y^\ast+g\,dY^\ast+dh\wedge Z^\ast+h\,dZ^\ast\\=&-YfX^\ast\wedge Y^\ast-ZfX^\ast\wedge Z^\ast-\kappa fY^\ast\wedge Z^\ast+XgX^\ast\wedge Y^\ast-ZgY^\ast\wedge Z^\ast+\\&+\kappa gX^\ast\wedge Z^\ast-hX^\ast\wedge Y^\ast+XhX^\ast\wedge Z^\ast+YhY^\ast\wedge Z^\ast\\=&\underbrace{-hX^\ast\wedge Y^\ast}_{d_0}+\underbrace{(Xg-Yf)X^\ast\wedge Y^\ast+XhX^\ast\wedge Z^\ast+YhY^\ast\wedge Z^\ast}_{d_{(1)}}+\\&+\underbrace{(\kappa g-Zf)X^\ast\wedge Z^\ast-(\kappa f+Zg)Y^\ast\wedge Z^\ast}_{d_{(2)}}\,.
\end{align*}
Finally, the direct sum decomposition associated to such finite filtration is given by
\begin{align*}
    \Omega^\bullet(G_\kappa)=K=K_{(0)}\oplus K_{(1)}\oplus K_{(2)}\oplus K_{(3)}\oplus K_{4}\,,
\end{align*}
where
\begin{itemize}
    \item $K_{(0)}=K^0=C^\infty(G_\kappa)$;
    \item $K_{(1)}=\spn_{C^\infty(G_\kappa)}\lbrace X^\ast,Y^\ast\rbrace$;
    \item  $K_{(2)}=\spn_{C^\infty(G_\kappa)}\lbrace Z^\ast,X^\ast\wedge Y^\ast\rbrace$;
    \item $K_{(3)}= \spn_{C^\infty(G_\kappa)}\lbrace X^\ast\wedge Z^\ast,Y^\ast\wedge Z^\ast\rbrace$;
    \item $K_{(4)}= \spn_{C^\infty(G_\kappa)}\lbrace X^\ast\wedge Y^\ast\wedge Z^\ast\rbrace$.
\end{itemize}

We will now proceed to show the computations of the spectral sequence associated to the multicomplex that arises from this particular filtration, just as described in the previous section.
Notice that in our computations, for all the pages $r>0$, we will compute the subspaces $Z_r^{p,\bullet}$ and $B_r^{p,\bullet}$, and use the isomorphism
\begin{align*}
    E_r^{p,\bullet}=E_r^{p,\bullet}(\mathrm{Tot}\,\Omega^\bullet(G_\kappa))=\frac{\mathcal{Z}_r^{p,\bullet}(\mathrm{Tot}\,\Omega^\bullet(G_\kappa))}{\mathcal{B}_r^{p,\bullet}(\mathrm{Tot}\,\Omega^\bullet(G_\kappa))}=\frac{Z_r^{p,\bullet}}{B_r^{p,\bullet}}\,.
\end{align*}

Since the filtration over the space of smooth forms is nontrivial up to $S=4$, we are going to present and discuss the quotients that appear in the spectral sequence construction for each $p=0,\ldots,4$. In order to make the computations more understandable, we are also going to summarise the overall information at each page in a more pictorial way with the use of diagrams. We should stress that we chose not to follow the standard convention of such diagrams, where at each page $r$ the quotient $E^{p,s}_r$ is placed at position $p$ along the horizontal axis and position $s$ along the vertical one. In this way, each quotient $E_r^{p,s}$ would be placed at the point $(p,s)$ on the Cartesian plane. Instead, in our diagrams the level of the filtration is encoded along the vertical axis, and the degree of the forms along the horizontal axis, so that each quotient $E^{p,s}_r$ will be placed at the point $(p+s,-p)$ on the Cartesian plane. Finally, in order to simplify the notation, we will denote by $\langle\cdot\rangle$ the span over the space of smooth functions $C^\infty(G_\kappa)$.

\setcounter{subsection}{-1}

\subsection{Page-zero quotients $E_0^{p,\bullet}$}

At the zeroth page of the spectral sequence for $r=0$, one has the quotients:
\begin{align*}
    E_0^{p,\bullet}=\frac{\mathcal{Z}_0^{p,\bullet}(\mathrm{Tot}\,\Omega^\bullet(G_\kappa))}{\mathcal{B}_0^{p,\bullet}(\mathrm{Tot}\,\Omega^\bullet(G_\kappa))}=\frac{\mathcal{F}_p\Omega^\bullet(G_\kappa)}{\mathcal{F}_{p+1}\Omega^\bullet(G_\kappa)}\,.
\end{align*}
$\mathbf{p=0}$
\begin{itemize}
    \item $E^{0,0}_0=\mathcal{F}_0\Omega^0(G_\kappa)/\mathcal{F}_1\Omega^0(G_\kappa)=K^0=\Omega^0(G_\kappa)=\langle 1\rangle$;
    \item $E_0^{0,1}=\mathcal{F}_0\Omega^1(G_\kappa)/\mathcal{F}_1\Omega^1(G_\kappa)=[0]_0$;
    \item $E_0^{0,2}=\mathcal{F}_0\Omega^2(G_\kappa)/\mathcal{F}_1\Omega^2(G_\kappa)=[0]_0$;
    \item $E_0^{0,3}=\mathcal{F}_0\Omega^3(G_\kappa)/\mathcal{F}_1\Omega^3(G_\kappa)=[0]_0$;
\end{itemize}
$\mathbf{p=1}$
\begin{itemize}
    \item $E_0^{1,-1}=\mathcal{F}_1\Omega^0(G_\kappa)/\mathcal{F}_2\Omega^0(G_\kappa)=0$;
    \item $E_0^{1,0}=\mathcal{F}_1\Omega^1(G_\kappa)/\mathcal{F}_2\Omega^1(G_\kappa)=K^1/K^1_{(2)}=\Omega^1(G_\kappa)/\langle Z^\ast\rangle=\langle [X^\ast]_0,[Y^\ast]_0\rangle$;
    \item $E_0^{1,1}=\mathcal{F}_1\Omega^2(G_\kappa)/\mathcal{F}_2\Omega^2(G_\kappa)= [0]_0$;
    \item $E_0^{1,2}=\mathcal{F}_1\Omega^3(G_\kappa)/\mathcal{F}_2\Omega^3(G_\kappa)=[0]_0$;
\end{itemize}
$\mathbf{p=2}$
\begin{itemize}
    \item $E^{2,-2}_0=\mathcal{F}_2\Omega^0(G_\kappa)/\mathcal{F}_3\Omega^0(G_\kappa)=0$;
    \item $E_0^{2,-1}=\mathcal{F}_2\Omega^1(G_\kappa)/\mathcal{F}_3\Omega^1(G_\kappa)=\langle Z^\ast\rangle$;
    \item $E_0^{2,0}=\mathcal{F}_2\Omega^2(G_\kappa)/\mathcal{F}_3\Omega^2(G_\kappa)=K^2/K^2_{(3)}=\langle \Omega^2(G_\kappa)/\langle X^\ast\wedge Z^\ast, Y^\ast\wedge Z^\ast\rangle=\langle [X^\ast\wedge Y^\ast]_0\rangle$;
    \item $E_0^{2,1}=\mathcal{F}_2\Omega^3(G_\kappa)/\mathcal{F}_3\Omega^3(G_\kappa)=[0]_0$;
\end{itemize}
$\mathbf{p=3}$
\begin{itemize}
    \item $E_0^{3,-3}=\mathcal{F}_3\Omega^0(G_\kappa)/\mathcal{F}_4\Omega^0(G_\kappa)=0$;
    \item $E_0^{3,-2}=\mathcal{F}_3\Omega^1(G_\kappa)/\mathcal{F}_4\Omega^1(G_\kappa)=0$;
    \item $E_0^{3,-1}=\mathcal{F}_3\Omega^2(G_\kappa)/\mathcal{F}_4\Omega^2(G_\kappa)=K^2_{(3)}=\langle X^\ast\wedge Z^\ast, Y^\ast\wedge Z^\ast\rangle$;
    \item $E_0^{3,0}=\mathcal{F}_3\Omega^3(G_\kappa)/\mathcal{F}_4\Omega^3(G_\kappa)=[0]_0$;
\end{itemize}
$\mathbf{p=4}$
\begin{itemize}
    \item $E_0^{4,-4}=\mathcal{F}_{4}\Omega^0(G_\kappa)=0$;
    \item $E_0^{4,-3}=\mathcal{F}_4\Omega^1(G_\kappa)=0$;
    \item $E_0^{4,-2}=\mathcal{F}_4\Omega^2(G_\kappa)=0$;
    \item $E_0^{4,-1}=\mathcal{F}_4\Omega^3(G_\kappa)=K^4=\Omega^3(G_\kappa)=\langle X^\ast\wedge Y^\ast\wedge Z^\ast\rangle$.
\end{itemize}
These computations can be easily summarised in a more visual way using the following diagram
\begin{align*}
    \begin{picture}(360,250)(0,0)
   \put(  0,  0){\line( 1, 0){360}}
   \put(  0,  0){\line( 0, 1){250}}
   \multiput(  0, 50)( 0,50){5}{\line( 1, 0){360}}
   \multiput( 90,  0)(90, 0){4}{\line( 0, 1){250}}
   \put(45,225){\makebox(0,0){$ 1$}}
   \put(45,175){\makebox(0,0){$ 0$}}
   \put(45,125){\makebox(0,0){$0$}}
   \put(135,225){\makebox(0,0){$0$}}
   \put(135,175){\makebox(0,0){$ [X^\ast]_0,[Y^\ast]_0$}}
   \put(135,125){\makebox(0,0){{ $ Z^\ast$}}}
   \put(135,75){\makebox(0,0){{ $ 0$}}}
   \put(225,175){\makebox(0,0){$0$}}
   \put(225,125){\makebox(0,0){$[X^\ast\wedge Y^\ast]_0$}}
   \put(225,75){\makebox(0,0){{ $X^\ast\wedge Z^\ast,Y^\ast\wedge Z^\ast$}}}
   \put(225,25){\makebox(0,0){{ $ 0$}}}
   \put(315,125){\makebox(0,0){{ $0$}}}
   \put(315,75){\makebox(0,0){{ $0$}}}
   \put(315,25){\makebox(0,0){{ $X^\ast\wedge Y^\ast\wedge Z^\ast$}}}
   \put(60,225){\vector(1,0){55}}
   \put(60,175){\vector(1,0){40}}
   \put(60,125){\vector(1,0){55}}
   \put(168,175){\vector(1,0){40}}
   \put(150,125){\vector(1,0){45}}
   \put(150,75){\vector(1,0){35}}
   \put(255,125){\vector(1,0){50}}
   \put(268,75){\vector(1,0){35}}
   \put(240,25){\vector(1,0){40}}
  \end{picture} 
\end{align*}
where we are also drawing the action of the operators $\partial_0=d_0\colon E_0^{p,\bullet}\to E^{p,\bullet}_0$.
Moreover, let us stress that such operator $\partial_0$ has nontrivial action only on $E_0^{2,-1}$, as
\begin{align*}
    \partial_0\colon E^{2,-1}_0\to E^{2,0}_0\ ,\ \partial_0(Z^\ast)=d_0(Z^\ast)=-X^\ast\wedge Y^\ast =-[X^\ast\wedge Y^\ast]_0\,.
\end{align*}
\subsection{Page-one quotients $E_1^{p,\bullet}$} 
Starting from page $r=1$, we are going to use the isomorphism
\begin{align*}
    E_1^{p,\bullet}=\frac{\mathcal{Z}_1^{p,\bullet}(\mathrm{Tot}\,\Omega^\bullet(G_\kappa))}{\mathcal{B}_1^{p,\bullet}(\mathrm{Tot}\,\Omega^\bullet(G_\kappa))}=\frac{Z_1^{p,\bullet}}{B_1^{p,\bullet}}\,,
\end{align*}
where
\begin{align*}
    Z_1^{p,\bullet}=&\lbrace\alpha\in K_{(p)}: d_0\alpha=0\rbrace\ \text{ and }
    B_1^{p,\bullet}=\lbrace\beta\in K_{(p)}: \exists\ \gamma\in K_{(p)}\text{ such that }d_0\gamma=\beta\rbrace\,.
\end{align*}
$\mathbf{p\neq 2}$\\ As pointed out before, we have that $d_0$ acts as the zero map on any $K_{(p)}$ with $p\neq 2$, and so $E_1^{p,\bullet}=E^{p,\bullet}_0$.\\
$\mathbf{p=2}$\\ The map $d_0\colon K^1_{(2)}\to K^2_{(2)}$
is an isomorphism, and so
\begin{align*}
    Z_1^{2,-1}&=\lbrace \alpha\in K^1_{(2)}: d_0\alpha=0\rbrace= 0 \ \text{ and }\ B_1^{2,-1}= 0\\
    Z_1^{2,0}=&\lbrace \alpha\in K^2_{(2)}: d_0\alpha=0\rbrace=K^2_{(2)}\ \text{ and }\ B_1^{2,0}=K^2_{(2)}\,.
\end{align*}
Therefore
\begin{itemize}
    \item $E^{2,-2}_1=E^{2,-2}_0=0$;
    \item $E_1^{2,-1}=Z_1^{2,-1}/B_1^{2,-1}=0$;
    \item $E_1^{2,0}=Z_1^{2,0}/B_1^{2,0}= [0]_1$;
    \item $E_1^{2,1}=E_0^{2,1}=[0]_0$.
\end{itemize}
By the properties of the spectral sequence, for any $r> 1$, $E_r^{2,\bullet}\subset E_1^{2,\bullet}$, and so $E^{2,\bullet}_r\cong 0$ for any $r\ge 1$.
\begin{align*}
    \begin{picture}(360,250)(0,0)
   \put(  0,  0){\line( 1, 0){360}}
   \put(  0,  0){\line( 0, 1){250}}
   \multiput(  0, 50)( 0,50){5}{\line( 1, 0){360}}
   \multiput( 90,  0)(90, 0){4}{\line( 0, 1){250}}
   \put(45,225){\makebox(0,0){$ 1$}}
   \put(45,175){\makebox(0,0){$ 0$}}
   \put(135,225){\makebox(0,0){$0$}}
   \put(135,175){\makebox(0,0){$ [X^\ast]_0,[Y^\ast]_0$}}
   \put(135,125){\makebox(0,0){{ $ 0$}}}
   \put(135,75){\makebox(0,0){{ $ 0$}}}
   \put(225,175){\makebox(0,0){$0$}}
   \put(225,125){\makebox(0,0){$[0]_1$}}
   \put(225,75){\makebox(0,0){{ $X^\ast\wedge Z^\ast,Y^\ast\wedge Z^\ast$}}}
   \put(225,25){\makebox(0,0){{ $ 0$}}}
   \put(315,75){\makebox(0,0){{ $0$}}}
   \put(315,25){\makebox(0,0){{ $X^\ast\wedge Y^\ast\wedge Z^\ast$}}}
   \put(58,220){\vector(1,-0.8){45}}
   \put(166,170){\vector(1,-0.8){45}}
   \put(147,120){\vector(1,-0.8){45}}
   \put(263,70){\vector(1,-1){35}}
  \end{picture} 
\end{align*}
In this diagram of page one, we are also drawing the action of the operators $\partial_1\colon E_1^{p,\bullet}\to E_1^{p+1,\bullet}$, which have nontrivial action only on $E_0^{0,0}$ and $E^{3,-1}_1$. Indeed
\begin{align*}
    \partial_1&\colon E^{0,0}_1\to E_1^{1,0}\ ,\ \partial_1[f]=[d_{(1)}f]\ , \ \forall\ f\in C^\infty(G_\kappa)\,,\\
    \partial_1\colon& E_1^{3,-1}\to E_1^{4,-1}\ ,\ \partial_1[\alpha]=[d_{(1)}\alpha]\ ,\ \forall\ \alpha\in Z^{3,-1}_1\,.
\end{align*}
In particular, for any $f\in C^\infty(G_\kappa)$, we have that
\begin{align*}
    \partial_1[f]=[d_{(1)}f]=[XfX^\ast+YfY^\ast]= XfX^\ast+YfY^\ast\ \mathrm{mod}\,\langle Z^\ast\rangle\,,
\end{align*}
while for any $\alpha=fX^\ast\wedge Z^\ast+gY^\ast\wedge Z^\ast\in Z_1^{3,-1}=Z_0^{3,-1}=K_{(3)}^2$, we have that
\begin{align*}
    \partial_1\big[fX^\ast \wedge Z^\ast+gY^\ast\wedge Z^\ast\big]=\big[d_{(1)}\big(fX^\ast \wedge Z^\ast+gY^\ast\wedge Z^\ast\big)\big]=(Xg-Yf) X^\ast\wedge Y^\ast\wedge Z^\ast\,.
\end{align*}
\subsection{Page-two quotients $E_2^{p,\bullet}$} By definition, the subspaces $Z_2^{p,\bullet},B_2^{p,\bullet}\subset K_{(p)}$ are given by
\begin{align*}
    Z_2^{p,\bullet}=&\lbrace \alpha\in  K_{(p)}: d_0\alpha=0\text{ and }\exists\ \beta_1\in K_{(p+1)}\text{ s.t. }d_1\alpha=d_0\beta_1\rbrace\\
    B_2^{p,\bullet}=&\lbrace\beta\in K_{(p)}: \exists\ \gamma_0\in K_{(p)}, \gamma_1\in K_{(p-1)}\text{ s.t. }\beta=d_0\gamma_0+d_{(1)}\gamma_1\text{ and }d_0\gamma_1=0\rbrace 
\end{align*}
$\mathbf{p=0}$\\
The only nontrivial space of forms in $K_{(0)}$ is the space of 0-forms, and so
\begin{align*}
    Z_2^{0,0}=\lbrace f\in C^\infty(G_\kappa)\mid d_0f=0\text{ and }\exists\, g\in K^0_{(1)}\text{ s.t. }d_{(1)}f=d_0 g\rbrace\,.
\end{align*}
Since $K^0_{(1)}=0$, the condition $f\in Z_2^{0,0}$ reads as $d_{(1)}f=0$, or in other words
\begin{align*}
    f\in Z_2^{0,0}\longleftrightarrow d_{(1)}f=Xf X^\ast+Yf Y^\ast=0\longleftrightarrow Xf\equiv Yf\equiv 0\,.
\end{align*}
Moreover, since $Z=[X,Y]$, we have that $Z_2^{0,0}=\mathbb{R}$ coincides with the constant functions, and so $E_2^{0,0}=Z_2^{0,0}=\mathbb{R}$. Finally, we know that the de Rham cohomology class in degree 0, $H^0(\Omega^\bullet(G_\kappa),d)$ is isomorphic to $\mathbb R$, and so by the properties of spectral sequences, we readily get that $E_r^{0,0}\cong H^0(G_\kappa,d)\cong\mathbb R$ for any $r\ge 2$.\\
$\mathbf{p=1}$\\ 
The only nontrivial space of forms in $K_{(1)}$ is the space of 1-forms, and so
\begin{align*}
    Z_2^{1,0}=\lbrace\alpha\in K_{(1)}^1: d_0\alpha=0\text{ and }\exists\,\beta_1\in K_{(2)}^1\text{ s.t. }d_{(1)}\alpha=d_0\beta_1\rbrace\,.
\end{align*}An arbitrary element $\alpha\in K_{(1)}^1$ with $\alpha=fX^\ast+gY^\ast$ satisfies $d_0\alpha=0$, moreover
\begin{align*}
    d_{(1)}\alpha=(Xg-Yf)X^\ast\wedge Y^\ast=d_0[(Yf-Xg)Z^\ast]\,,\ \text{ where }(Yf-Xg)Z^\ast\in K^1_{(2)}
\end{align*}
and so $Z_2^{1,0}=K^1_{(1)}$.
On the other hand, 
\begin{align*}
    B_2^{1,0}=\lbrace \beta\in K_{(1)}^1: \exists \,\gamma_0\in K_{(1)}^0,\,\gamma_1\in K_{(0)}^0\text{ s.t. }\beta=d_0\gamma_0+d_{(1)}\gamma_1\text{ and }d_0\gamma_1=0\rbrace
\end{align*}
and since $K_{(1)}^0=0$, we have that an element $\beta=fX^\ast+gY^\ast\in B_2^{1,0}$ if there exists $h\in K_{(0)}^0$ such that 
\begin{align*}
    \beta=d_{(1)}h\text{ and }d_0h=0 \longleftrightarrow \beta= d_{(1)}h=Xh X^\ast+Yh Y^\ast\,.
\end{align*}
Therefore, we have that $E_2^{1,0}=Z_2^{1,0}/B^{1,0}_2=K_{(1)}^1/d_{(1)}K_{(0)}^0=\langle [X^\ast,Y^\ast]_2\rangle$.\\
$\mathbf{p=2}$\\
As already pointed out in the computations of the previous page, we have that $E_2^{2,\bullet}\cong 0$.\\
$\mathbf{p=3}$\\
The only nontrivial space of forms in $K_{(3)}$ is the space of 2-forms, and so
\begin{align*}
    Z_2^{3,-1}=\lbrace \alpha\in K_{(3)}^2: d_0\alpha=0\text{ and }\exists\,\beta_1\in K_{(4)}^2\text{ s.t. } d_{(1)}\alpha=d_0\beta_1\rbrace\,.
\end{align*}
An arbitrary element $\alpha\in K_{(3)}^2$ with $\alpha=fX^\ast\wedge Z^\ast+gY^\ast\wedge Z^\ast$ satisfies $d_0\alpha=0$, moreover, since $K_{(4)}^2=0$, we have that $\alpha\in Z_2^{3,-1}$ if $d_{(1)}\alpha=0$. Since also $d_{(2)}\alpha=0$ for any $\alpha\in K_{(3)}^2$, we have that $d\alpha=d_0\alpha+d_{(1)}\alpha+d_{(2)}\alpha=d_{(1)}\alpha$, and so the space $Z_2^{3,-1}$ coincides with the space of closed forms in $K_{(3)}^2$.
On the other hand, 
\begin{align*}
    \beta\in B_2^{3,-1}\longleftrightarrow \exists\ \gamma_0\in K_{(3)}^1,\gamma_1\in K_{(2)}^1\text{ s.t. }\beta=d_0\gamma_0+d_{(1)}\gamma_1\text{ and }d_0\gamma_1=0\,.
\end{align*}
We know that $K_{(3)}^1=0$, however, for any form $\gamma_1=f Z^\ast\in K_{(2)}^1$, we have $d_0\gamma_1=-fX^\ast\wedge Y^\ast\neq 0$, and so $B_2^{3,-1}=0$. Therefore, $E_2^{3,-1}=Z_2^{3,-1}/B_2^{3,-1}=Z_2^{3,-1}=\lbrace \alpha\in K_{(3)}^2:d\alpha=0\rbrace=\langle [X^\ast\wedge Z^\ast,Y^\ast\wedge Z^\ast]_2\rangle$.\\
$\mathbf{p=4}$\\
Finally, the only nontrivial space of forms in $K_{(4)}$ is the space of 3-forms, and so
\begin{align*}
    Z^{4,-1}_2=\lbrace\alpha\in K^3_{(4)}: d_0\alpha=0\text{ and }\exists\,\beta_1\in K_{(5)}^3\text{ s.t. }d_{(1)}\alpha=d_0\beta_1\rbrace\,,
\end{align*} and since $K_{(5)}^3=K_{(5)}^4=0$, we have that $ Z_2^{4,-1}=K_{(4)}^3$. On the other hand, the condition $\beta\in B_2^{4,-1}$ reads
\begin{align*}
    \beta\in B_2^{4,-1}\longleftrightarrow\exists\, \gamma_0\in K_{(4)}^2,\gamma_1\in K_{(3)}^2\text{ s.t. }\beta=d_0\gamma_0+d_{(1)}\gamma_1\text{ and }d_0\gamma_1=0\,.
\end{align*}
Since $K_{(4)}^2=0$, and $d_0\gamma_1=0$ for any $\gamma_1\in K_{(3)}^2$, we get $B^{4,-1}_2=\lbrace\beta\in K_{(4)}^3: \exists\, \gamma_1\in K_{(3)}^2\text{ s.t. }d_{(1)}\gamma_1=\beta\rbrace=\lbrace (Xg-Yf)X^\ast\wedge Y^\ast\wedge Z^\ast\rbrace$. \\
Therefore, 
$E_2^{4,-1}=Z_2^{4,-1}/B_2^{4,-1}=K_{(4)}^3/ d_{(1)}K_{(3)}^2=[X^\ast\wedge Y^\ast\wedge Z^\ast]_2$.\\
It should be noted that any $\alpha\in Z_2^{4,-1}=K_{(4)}^3$ is a closed form, since $d\alpha=0$. If $\alpha$ is also exact, then there exists $\beta=\beta_{(2)}+\beta_{(3)}\in K^2$  such that $d\beta=\alpha$. Here, $\beta_{(2)}\in K_{(2)}^2$ and $\beta_{(3)}\in K_{(3)}^2$, where $\beta_{(2)}=d_0\gamma_{(2)}$ for some $\gamma_{(2)}\in K_{(2)}^1$ since the map $d_0\colon K_{(2)}^1\to K_{(2)}^2$ is an isomorphism. Therefore
\begin{align*}
    \alpha=d\beta=d_{(2)}\beta_{(2)}+d_{(1)}\beta_{(3)}=d_{(2)}d_0\gamma_{(2)}+d_{(1)}\beta_{(3)}=-d_{(1)}^2\gamma_{(2)}+d_{(1)}\beta_{(3)}\,.
\end{align*}
Here, in the final equality, we use the fact that $(K,d)$ is a multicomplex and $d_{(0)}d_{(2)}+d_{(1)}^2+d_{(2)}d_{(0)}=0$, and that $d_{(2)}\colon K_{(2)}^1\to K_{(4)}^2$ is the zero map. Hence, for any exact form $\alpha\in Z_2^{4,-1}$ there exists $\tilde\beta=\beta_{(3)}-d_{(1)}\gamma_{(2)}\in K_{(3)}^2$ such that $\alpha=d_{(1)}\tilde\beta$. In other words, $E_{2}^{4,-1}$ is isomorphic to $H^3(\Omega^\bullet(G_\kappa),d)$, the de Rham cohomology class of $G_\kappa$ in degree 3. In the case where $H^3(\Omega^\bullet(G_\kappa),d)=0$, for example in the first Heisenberg group when $\kappa=0$, we have $E_2^{4,-1}\cong 0$, whereas if $H^3(\Omega^\bullet(G_\kappa),d)\cong \mathbb R$, for example when $\kappa >0$, then $E_2^{4,-1}\cong\mathbb R$. Following the same reasoning as in the case of $\mathbf{p=0}$, we also have that $E^{4,-1}_r\cong H^3(\Omega^\bullet(G_\kappa),d)$ for any $r\ge 2$.

\begin{align*}
    \begin{picture}(360,250)(0,0)
   \put(  0,  0){\line( 1, 0){360}}
   \put(  0,  0){\line( 0, 1){250}}
   \multiput(  0, 50)( 0,50){5}{\line( 1, 0){360}}
   \multiput( 90,  0)(90, 0){4}{\line( 0, 1){250}}
   \put(45,225){\makebox(0,0){$ \mathbb{R}$}}
   \put(45,175){\makebox(0,0){$ 0$}}
   \put(135,225){\makebox(0,0){$0$}}
   \put(135,175){\makebox(0,0){$ [X^\ast,Y^\ast]_2$}}
   \put(135,125){\makebox(0,0){{ $ 0$}}}
   \put(135,75){\makebox(0,0){{ $ 0$}}}
   \put(225,175){\makebox(0,0){$0$}}
   \put(225,125){\makebox(0,0){$[0]_1$}}
   \put(223,75){\makebox(0,0){{ $[X^\ast\wedge Z^\ast,Y^\ast\wedge Z^\ast]_2$}}}
   \put(225,25){\makebox(0,0){{ $ 0$}}}
   \put(315,75){\makebox(0,0){{ $0$}}}
   \put(315,25){\makebox(0,0){{ $[X^\ast\wedge Y^\ast\wedge Z^\ast]_2$}}}
   \put(58,220){\vector(1,-1.3){65}}
   \put(160,168){\vector(1,-1.6){50}}
  \end{picture} 
\end{align*}
In this diagram of page two, we are showing the action of the operators $\partial_2\colon E_2^{p,\bullet}\to E_2^{p+2,\bullet}$. The only nontrivial map of such operators is $\partial_2\colon E_2^{1,0}\to E_2^{3,-1}$, so for any $\alpha\in Z_2^{1,0}=K_{(1)}^1$ 
\begin{align*}
    \partial_2\big[\alpha\big]=\big[d_{(2)}\alpha-d_{(1)}\beta_1\big]\ \text{ for some }\beta_1\in K_{(2)}^1\text{ s.t. }d_{(1)}\alpha=d_0\beta_1\,.
\end{align*}
As already pointed out before, if we fix $\alpha=fX^\ast+g Y^\ast$, it is sufficient to take $\beta_1=(Yf-Xg)Z^\ast\in K_{(2)}^1$ in order to obtain $d_{(1)}\alpha=d_0\beta_1$, hence
\begin{align*}
    \partial_2\big[fX^\ast+gY^\ast\big]=&\big[d_{(2)}(fX^\ast+gY^\ast)-d_{(1)}(Yf-Xg)Z^\ast\big]\\=&\big[(\kappa g-Zf)X^\ast\wedge Z^\ast-(\kappa f+Zg)Y^\ast\wedge Z^\ast+\\&-X(Yf-Xg)X^\ast\wedge Z^\ast-Y(Yf-Xg)Y^\ast\wedge Z^\ast\big]\\=&(\kappa g-Zf-XYf+X^2g)X^\ast\wedge Z^\ast+(YXg-Zg-Y^2f-\kappa f)Y^\ast\wedge Z^\ast
\end{align*}
One can also readily check that $\partial_2[\alpha]$ is an element of $E_2^{3,-1}$, since
\begin{align*}
    X(YXg-Zg-Y^2f-\kappa f)=Y(\kappa g-Zf-XYf+X^2g)\ ,\ \forall\, f,g\in C^\infty(G_\kappa)\,.
\end{align*}
Moreover, the differential $\partial_2$ is well-defined on the quotient $E_2^{1,0}$ since for any $\beta\in B_2^{1,0}$ we have $\beta=XhX^\ast+YhY^\ast$ for some $h\in C^\infty(G_\kappa)$, and so
\begin{align*}
    \partial_2[\beta]=&\big[d_{(2)}(XhX^\ast+YhY^\ast)-d_{(1)}(YXh-XYh)Z^\ast\big]\\
    =&\big[(\kappa Yh-ZXh)X^\ast\wedge Z^\ast-(\kappa Xh+ZYh)Y^\ast\wedge Z^\ast+XZhX^\ast\wedge Z^\ast+YZhY^\ast\wedge Z^\ast\big]=[0]\,.
\end{align*}
\subsection{Page-three quotients $E_3^{p,\bullet}$} By definition, we have that the subspaces $Z_3^{p,\bullet},B_3^{p,\bullet}\subset K_{(p)}$ are given by
\begin{align*}
    Z_3^{p,\bullet}=&\lbrace\alpha\in K_{(p)}\mid d_0\alpha=0\,,\,\exists\ \beta_1\in K_{(p+1)},\beta_2\in K_{(p+2)}\text{ s.t. }d_{(1)}\alpha=d_0\beta_1\,,\,d_{(2)}\alpha=d_0\beta_2+d_{(1)}\beta_1\rbrace\\
    B_3^{p,\bullet}=&\lbrace\beta\in K_{(p)}\mid \exists\ \gamma_i\in K_{(p-i)},\, i=0,1,2\text{ s.t. }\beta=\sum_{i=0}^2d_{(i)}\gamma_i\text{ and }d_0\gamma_1+d_{(1)}\gamma_2=d_0\gamma_2=0\rbrace\,.
\end{align*}
$\mathbf{p=0}$\\
As already pointed out in the computations of the previous page, we have that $E^{0,0}_3=E^{0,0}_2=\mathbb{R}$.\\
$\mathbf{p=1}$\\
We have already shown that an arbitrary element $\alpha=fX^\ast+gY^\ast\in K_{(1)}^1$ belongs to $Z^{1,0}_2$ since $d_0\alpha=0$ and one can take $\beta_1=(Yf-Xg)Z^\ast\in K_{(2)}^1$ such that $d_{(1)}\alpha=(Xg-Yf)X^\ast\wedge Y^\ast=d_0\beta_1$. Hence $\alpha\in Z^{1,0}_3$ if there exists $\beta_2\in K_{(3)}^1$ such that $d_{(2)}\alpha=d_0\beta_2+d_{(1)}\beta_1$, but since $K_{(3)}^1=0$, this means that $$\alpha\in Z^{1,0}_3\longleftrightarrow \alpha\in K_{(1)}^1\text{ and }d_{(2)}\alpha=d_{(1)}\beta_1\text{ with }d_{(1)}\alpha=d_{0}\beta_1\,.$$
On the other hand, $\beta\in B_3^{1,0}$ if there exist $\gamma_i\in K_{(1-i)}^0$, with $i=0,1,2$ such that $\beta=d_0\gamma_0+d_{(1)}\gamma_1+d_{(2)}\gamma_2$ and $d_0\gamma_1+d_{(1)}\gamma_2=d_0\gamma_2=0$. Since $K_{(-1)}^0=K_{(1)}^0=0$, we have that $\beta\in B_3^{1,0}$ if $\exists\, h\in C^\infty(G_\kappa)$ such that $\beta=d_{(1)}h$ and $d_0h=0$, or in other words $B_3^{1,0}=B_2^{1,0}$.\\
To show that $E^{1,0}_3=Z_3^{1,0}/B_3^{1,0}\cong 0$, let us first consider the exterior differential of an arbitrary form $\alpha\in Z_3^{1,0}$:
\begin{align*}
    d\alpha=d_{(0)}\alpha+d_{(1)}\alpha+d_{(2)}\alpha=0+d_0\beta_1+d_{(2)}\alpha=d\beta_1-d_{(1)}\beta_1-d_{(2)}\beta_1+d_{(2)}\alpha=d\beta_1
\end{align*}
since $d_{(2)}\colon K_{(2)}^1\to K_{(4)}^2$ is the zero map. This means that the form $d(\alpha-\beta_1)=0$, that is $\alpha-\beta_1$ is a closed form. We know that $H^1(\Omega^\bullet(G_\kappa),d)=0$ for any $\kappa$, and so $\alpha-\beta_1$ is also exact, that is $\exists\, h\in C^\infty(G_\kappa)$ such that
\begin{align*}
    \alpha-\beta_1=dh=d_{(1)}h+d_{(2)}h\ \longleftrightarrow \ \alpha=d_{(1)}h\in K_{(1)}^1\text{ and }\beta_1=d_{(2)}h\in K_{(2)}^1\,.
\end{align*}
We have therefore shown that for any $\alpha\in Z_3^{1,0}$, there exists $h\in K_{(0)}^0$ such that $d_{(1)}h=\alpha$, that is $Z_3^{1,0}=B_3^{1,0}$.\\
$\mathbf{p=3}$\\
By definition, we have that $\alpha\in Z_3^{3,-1}$ if $d_0\alpha=0$ and there exist $\beta_1\in K_{(4)}^2$ and $\beta_2\in K_{(5)}^2$ such that $d_{(1)}\alpha=d_0\beta_1$ and $d_{(2)}\alpha=d_0\beta_2+d_{(1)}\beta_1$. Again, since $K_{(4)}^2=K_{(5)}^2=0$ and $d_0\alpha=0$ for any $\alpha\in K_{(3)}^2$, we have that $\alpha\in Z_3^{3,-1}$ if $d_{(1)}\alpha=0$. In other words, $Z^{3,-1}_3=Z^{3,-1}_2$ is the space of closed forms in $K_{(3)}^2$. On the other hand, we have $\beta\in B^{3,-1}_3$ if there exist $\gamma_i\in K_{(3-i)}^1$ with $i=0,1,2$ such that $\beta=d_0\gamma_0+d_{(1)}\gamma_1+d_{(2)}\gamma_2$ and $d_0\gamma_1+d_{(1)}\gamma_2=d_0\gamma_2=0$. Since $K_{(3)}^1=0$, we have
\begin{align*}
    \beta\in B^{3,-1}_3\longleftrightarrow \exists\,\gamma_1\in K_{(2)}^1,\,\gamma_2\in K_{(1)}^1\text{ s.t. } \beta = d_{(1)}\gamma_1+d_{(2)}\gamma_{2}\text{ and }d_0\gamma_1+d_{(1)}\gamma_2=d_0\gamma_2=0\,.
\end{align*}We now want to show that $B_3^{3,-1}=Z_3^{3,-1}$, and hence $E^{3,-1}_3\cong 0$. Since $H^2(\Omega^\bullet(G_\kappa),d)=0$ for any $\kappa$, we have that any $\alpha\in Z_3^{3,-1}$ is a closed and hence exact form, that is $\exists\,\beta=\beta_1+\beta_2\in K^1$ such that $d\beta=\alpha$. Here, we are taking $\beta_1\in K_{(1)}^1$ and $\beta_2\in K_{(2)}^1$, and so
\begin{align*}
    d\beta=d\beta_{(1)}+d\beta_{(2)}=d_{(1)}\beta_{(1)}+d_{(2)}\beta_{(1)}+d_{(0)}\beta_2+d_{(1)}\beta_2=\alpha\in K_{(3)}^2
\end{align*}
which implies that $d_{(1)}\beta_1+d_{0}\beta_2=0$ since it is an element of $ K_{(2)}$, and $d_0\beta_1=0$ since $d_0\colon K_{(1)}^1\to K_{(1)}^2$ is the zero map.
For any $\alpha\in Z_3^{3,-1}$, it is therefore sufficient to take $\gamma_2=\beta_1\in K_{(1)}^1$ and $\gamma_{1}=\beta_2\in K_{(2)}^1$ to show that $\alpha\in B^{3,-1}_3$.

$\mathbf{p=4}$\\
As already pointed out in the computations of the previous page, we have that $E_3^{4,-1}=E_3^{4,-1}\cong H^3(\Omega^\bullet(G_\kappa),d)$.

\bibliographystyle{abbrv}
\bibliography{Bibliography}

\end{document}